\documentclass[11pt,reqno]{amsart}
\usepackage{amsmath,amsfonts,amssymb,amsthm,amscd,comment,euscript}
\usepackage[all]{xy}
\usepackage{graphicx}
\usepackage{enumerate} 
\usepackage[colorlinks=true]{hyperref} 
\usepackage[usenames,dvipsnames]{xcolor}
\usepackage{enumitem}

\usepackage{bbm}

\xymatrixcolsep{1.9pc}                          
\xymatrixrowsep{1.9pc}
\newdir{ >}{{}*!/-5\pt/\dir{>}}                  

 \calclayout

\swapnumbers

\theoremstyle{plain}
\newtheorem{lem}{Lemma}[section]
\newtheorem{cor}[lem]{Corollary}
\newtheorem{prop}[lem]{Proposition}
\newtheorem{thm}[lem]{Theorem}

\newtheorem*{tthmm}{Theorem}

\theoremstyle{definition}

\newtheorem{rem}[lem]{Remark}
\newtheorem{dfn}[lem]{Definition}

\newtheorem{defs}[lem]{Definition}

\renewcommand{\phi}{\varphi}
\renewcommand{\leq}{\leqslant}
\renewcommand{\geq}{\geqslant}
\renewcommand{\epsilon}{\varepsilon}

\renewcommand{\kappa}{\varkappa}

\DeclareMathOperator{\spec}{Spec}
 
\DeclareMathOperator{\proj}{proj}

\DeclareMathOperator{\Sets}{Sets} 
 
\DeclareMathOperator{\mot}{mot}

 \DeclareMathOperator{\cof}{cof}

 \DeclareMathOperator{\Ab}{Ab}

 \DeclareMathOperator{\nis}{\mathsf{nis}}

\newcommand{\cc}{\mathcal}
\newcommand{\bb}{\mathbb}

\newcommand{\wh}{\widehat}

\newcommand{\SH}{SH}
\newcommand{\eff}{\mathsf{eff}}
\newcommand{\veff}{\mathsf{veff}}
\newcommand{\fr}{\mathsf{fr}}

\newcommand{\pt}{\mathsf{pt}}
\newcommand{\ev}{\mathsf{ev}}

\newcommand{\Shv}{\mathsf{Shv}}
\newcommand{\Cor}{\mathcal{A}}
\newcommand{\Smk}{\mathsf{Sm}_k}
\newcommand{\Psh}{\mathsf{Psh}}
\newcommand{\Sm}{\cc Sm}
\newcommand{\ShvA}{\Shv(\Cor)}
\newcommand{\PshA}{\Psh(\Cor)}
\newcommand{\Set}{\mathrm{Set}}

\newcommand{\Gmn}[1]{\mathbb{G}_m^{\wedge #1}}

\newcommand{\inthom}[3]{\underline{\mathrm{Hom}}_{#1}(#2,#3)}
\newcommand{\semilocalsimplex}{\wh{\Delta}^{\bullet}_{K/k}}

\newcommand{\Gm}{\mathbb{G}_m}
\newcommand{\colim}[1]{\underset{#1}{\mathsf{colim}}}

\newcommand{\Ch}{\mathsf{Ch}}
\newcommand{\Sp}{\mathsf{Sp}}
\newcommand{\SpGm}{\Sp_{\Gm}}

\newcommand{\EM}{\mathrm{EM}}

\newcommand{\motive}[1]{M_{\Cor}^{\Gm}(#1)}

\newcommand{\otimesDK}{\underset{DK}{\otimes}}

\newcommand{\GammaSpc}{\Gamma\mathrm{Spc}^{sp}}

\newcommand{\mathscr}{\cc}
\begin{document}

\footskip30pt

\baselineskip=1.1\baselineskip

\title{Rational Enriched Motivic Spaces}
\address{Department of Mathematics, Swansea University, Fabian Way, Swansea SA1 8EN, UK}
\email{peterjbonart@gmail.com}

\author{Peter Bonart}

\begin{abstract}
Enriched motivic $\Cor$-spaces are introduced and studied in this paper, where $\cc A$ is an
additive category of correspondences. They are linear counterparts of motivic $\Gamma$-spaces.
It is shown that rational special enriched motivic $\widetilde{\mathrm{Cor}}$-spaces recover connective motivic
bispectra with rational coefficients, where $\widetilde{\mathrm{Cor}}$ is the category of Milnor--Witt correspondences.
\end{abstract}

\thanks{Supported by the Swansea Science Doctoral Training Partnerships, and the Engineering
and Physical Sciences Research Council (Project Reference: 2484592)}

\maketitle
\thispagestyle{empty}
\pagestyle{plain}

\tableofcontents

\section{Introduction}

In his celebrated paper~\cite{S} Segal introduced $\Gamma$-spaces and showed that they yield
infinite loop spaces. In~\cite{BFgamma} Bousfield and Friedlander defined a model category structure for $\Gamma$-spaces
and showed that its homotopy category recovers connective $S^1$-spectra. They also showed that fibrant
objects in this model category are given by very special $\Gamma$-spaces.

Garkusha, Panin and \O stv\ae r~\cite{garkusha2019framed} have recently introduced and studied motivic $\Gamma$-spaces.
They are $\cc M$-enriched functors in two variables
   $$\cc X:\Gamma^{op}\boxtimes\Smk{}_{,+}\to\cc M,$$
where $\cc M$ is the category of pointed motivic spaces and $\Smk{}_{,+}$ is the $\cc M$-category of framed
correspondences of level 0. Special and very special motivic $\Gamma$-spaces are defined in~\cite{garkusha2019framed}
as $\cc M$-enriched functors 
   $$\cc X:\Gamma^{op}\boxtimes\Smk{}_{,+}\to\cc M^{\fr}$$
satisfying several axioms, where $\cc M^{\fr}$ is the $\cc M$-category of pointed motivic spaces with framed correspondences. The axioms 
are a combination of Segal's axioms and axioms reflecting basic properties of framed motives of algebraic varieties in the sense 
of Garkusha--Panin~\cite{garkusha2021framedmotives} (see~\cite{garkusha2019framed} for details).

Inspired by~\cite{garkusha2019framed} we introduce and study additive versions for motivic $\Gamma$-spaces in this paper. We start with a
reasonable additive category of correspondences $\Cor$, such that the exponential characteristic of the base field $k$
is invertible in $\Cor$, and replace $\cc M$ by the closed symmetric monoidal 
Grothendieck category $\Delta^{op}\ShvA$
of simplicial Nisnevich sheaves with $\Cor$-transfers. The $\cc M$-category $\Smk{}_{,+}$ is replaced here by a $\Delta^{op}\ShvA$-category
$\Sm$ whose objects are those of $\Smk$.

We define enriched motivic $\Cor$-spaces as objects of the Grothendieck category of $\Delta^{op}\ShvA$-enriched functors $[\Sm,\Delta^{op}\ShvA]$.
Special enriched motivic $\Cor$-spaces are defined similarly to special motivic $\Gamma$-spaces with slight modifications due to the additive context
(see Definition~\ref{AspaceDef} for the full list of axioms).
In particular, the category $\Gamma^{op}$ is redundant in this context (see Section~\ref{relatedtogamma}).

The category $[\Sm,\Delta^{op}\ShvA]$ comes equipped with a local and a motivic model structure. 
Denote the model categories by $[\Sm,\Delta^{op}\ShvA]_{\nis}$ and $[\Sm,\Delta^{op}\ShvA]_{\mot}$ respectively (see Section~\ref{motmodel}).
Let $\cc D([\Sm,\Delta^{op}\ShvA])$ be the homotopy category of $[\Sm,\Delta^{op}\ShvA]_{\nis}$. 
Define $\mathrm{Spc}_{\Cor}[\Sm]$ as the full subcategory of 
$\cc D([\Sm,\Delta^{op}\ShvA])$ consisting of special enriched motivic $\Cor$-spaces.
It is worth mentioning that $\cc D([\Sm,\Delta^{op}\ShvA])$ is equivalent to the full subcategory of connective 
chain complexes in the derived category $D([\Sm,\ShvA])$ of the Grothendieck category
$[\Sm,\ShvA]$. Thus $\mathrm{Spc}_{\Cor}[\Sm]$ can also be regarded
as a full subcategory of $D([\Sm,\ShvA])$, so that it can be studied by methods of classical homological algebra.

The following result is reminiscent of Bousfield--Friedlander's theorem mentioned above for classical $\Gamma$-spaces (see Theorem~\ref{corSpcAembedding}).

\begin{tthmm}
	 The category $\mathrm{Spc}_{\Cor}[\Sm]$ is equivalent to the homotopy category of the model category $[\Sm,\Delta^{op}\ShvA]_{\mot}$.
The fibrant objects of $[\Sm,\Delta^{op}\ShvA]_{\mot}$ are the pointwise locally fibrant special enriched motivic $\Cor$-spaces.
\end{tthmm}

As applications of the preceding theorem we recover connective motivic bispectra 
with rational coefficients $\SH(k)_{\bb Q, \geq 0}$ (respectively very effective 
motivic bispectra with rational coefficients $\SH^{\veff}(k)_{\bb Q}$) from
special rational enriched motivic $\Cor$-spaces $\mathrm{Spc}_{\Cor}[\Sm]$
(respectively very effective rational enriched motivic $\Cor$-spaces $\mathrm{Spc}^{\veff}_{\Cor}[\Sm]$) ---
see Theorems~\ref{thmSHgeq} and~\ref{finish}. 
Here we take $\cc A$ to be the category of finite Milnor--Witt correspondences with
rational coefficients $\widetilde{Cor}\otimes\bb Q$.

\begin{tthmm}
	The $(S^1,\Gm)$-evaluation functor induces equivalences of categories
	$$ev_{S^1,\Gm}: \mathrm{Spc}_{\widetilde{\mathrm{Cor}}, \bb Q}[\Sm] \rightarrow \SH(k)_{\bb Q,\geq 0}.$$
and
	$$ev_{S^1,\Gm}: \mathrm{Spc}^{\veff}_{\widetilde{\mathrm{Cor}}, \bb Q}[\Sm] \rightarrow \SH^{\veff}(k)_{\bb Q}.$$
\end{tthmm}

In particular, the preceding theorem makes $SH(k)_{\bb Q}$ more amenable to methods of homological algebra.
It is also worth mentioning that the results of this paper are based on the techniques developed in \cite{bonart2022paper1}.

The results of the paper were first presented at the Conference on Motivic and Equivariant Topology in May 2023 (Swansea, UK).
The author expresses his gratitude to his supervisor Prof. Grigory Garkusha
whose patience and keen insight have been indispensable throughout this work.

\subsubsection*{\bf Notation}
Throughout the paper we use the following notation.\\

\begin{tabular}{l|l}
	$k$ & field of exponential characteristic $p$\\
	$\Smk$ & smooth separated schemes of finite type over $k$ \\
	$\Cor$ & symmetric monoidal additive $V$-category of correspondences\\ 
	$\PshA$ & presheaves of abelian groups on $\Cor$ \\
	$\ShvA$ & Nisnevich sheaves of abelian groups on $\Cor$ \\
	$DM_{\Cor}$ & triangulated category of big motives with $\Cor$-correspondences\\
	$SH(k)$ & stable motivic homotopy category over $k$\\
	$M_{\Cor}(X)$ & $\Cor$-motive of $X\in\Smk$\\
	$\mathscr M$ & category of motivic spaces\\
	$f\mathscr M$ & category of finitely presented motivic spaces
\end{tabular}

Also, we assume that $0$ is a natural number.

\section{Preliminaries}
Let $\Smk$ be the category of smooth separated schemes of finite type over a filed $k$. 
Throughout this paper we work with an additive category of correspondences $\Cor$
that is symmetric monoidal and satisfies the strict $V$-property and the cancellation property 
in the sense of \cite{garkusha2019compositio}. Basic examples are given by the categories of finite
correspondences $Cor$ or Milnor--Witt correspondences $\widetilde{Cor}$.
We also assume that in $\Cor$ the exponential characteristic $p$ of $k$ is invertible. 
Note that for any additive category of correspondences $\Cor$ we can form an 
additive category of correspondences $\Cor[1/p]$ in which $p$ is invertible by 
tensoring all morphism groups of $\Cor$ with $\bb Z[1/p]$. 
Let $\ShvA$ be the category of Nisnevich sheaves on $\Cor$ with values in abelian groups.

We shall adhere to the following notations from \cite{garkusha2019compositio}.
Let $\Sp_{S^1,\bb G_m}(k)$ denote the category of symmetric $(S^1,\bb
G_m)$-bispectra, where the $\bb G_m$-direction is associated with
the pointed motivic space $(\bb G_m,1)$. It is equipped with a
stable motivic model category structure. Denote by $SH(k)$
its homotopy category. The category $SH(k)$ has a closed symmetric
monoidal structure with monoidal unit being the motivic sphere
spectrum $\bb S$. Given $p>0$, the
category $\Sp_{S^1,\bb G_m}(k)$ has a further model structure whose
weak equivalences are the maps of bispectra $f:X\to Y$ such that
the induced map of bigraded Nisnevich sheaves
$f_*:\underline{\pi}_{*,*}^{\bb A^1}(X)\otimes\bb Z[1/p]\to\underline{\pi}_{*,*}^{\bb A^1}(Y)\otimes\bb Z[1/p]$
is an isomorphism. In what follows we denote its homotopy category by
$SH(k)[1/p]$. The category $SH(k)_{\bb Q}$ is defined in a similar fashion.

Following~\cite{bonart2022paper1},
we define a $\ShvA$-enriched category $\Sm$, whose objects are those of $\Smk$, and whose morphism sheaves are defined by
$$\Sm(X,Y) := \inthom{\ShvA}{\Cor(-,X)_{\nis}}{\Cor(-,Y)_{\nis}}.$$

In Section \ref{sectionmodelstructures} we will define a natural local model structure on $\Delta^{op}\ShvA$.
Weak equivalences in this model structure are the local equivalences.

	According to \cite[Theorem 4.3.12]{clarke2019grothendieck}, if $\cc G$ is a Grothendieck category with a generator $G$, then 
	the category of simplicial objects $\Delta^{op}\cc G$ in $\cc G$ is also Grothendieck and
	the set 
	$\{ G \otimes \Delta[n] \mid n \geq 0 \} $
	is a family of generators for $\Delta^{op}\cc G$. 
	In particular, a family of generators for the Grothendieck category $\Delta^{op}\ShvA$ is given by the set $$\{ \Cor(-,X)_{\nis} \otimes \Delta[n] \mid X \in \Smk, n \geq 0 \}.$$
	Also, the category of enriched functors $[\Sm,\ShvA]$ is Grothendieck by~\cite{AlGG}. Its family of generators is given by
	$\{ \Sm(X,-)\otimes_{\ShvA}\Cor(-,Y)_{\nis}\mid X,Y \in \Smk \}.$ Hence $\Delta^{op}[\Sm,\ShvA]$ is Grothendieck by~\cite{clarke2019grothendieck}. 
	Its family of generators is given by
	$\{ \Sm(X,-)\otimes_{\ShvA}\Cor(-,Y)_{\nis}\otimes \Delta[n]\mid X,Y \in \Smk, n \geq 0 \}.$

Note that $\Delta^{op}[\Sm,\ShvA]$ and $[\Sm,\Delta^{op}\ShvA]$ are equivalent, and we will freely pass back and forth between the two.

\begin{dfn} \label{AspaceDef}
	An {\it enriched motivic $\Cor$-space\/} is an object of the Grothendieck category $\Delta^{op}[\Sm,\ShvA]$. 
	Similarly to \cite[Axioms 1.1]{garkusha2019framed}, an enriched motivic $\Cor$-space $\mathcal{X}$ is said to be \textit{special} if it satisfies the following axioms:
	\begin{enumerate}
		\item
		For all $n\geq 0$ and $U\in \Smk$ the presheaf of homotopy groups
		$
		V
		\longmapsto
		\pi_n(\cc X (U))(V)
		$
		is ${\bb A}^{1}$-invariant.
		\vspace{0.05in}
		
		\item (Cancellation)
		Let $\Gmn{1}$ denote the direct summand of the $1$-section 
		$\Cor(-,pt)_{\nis} \longrightarrow\Cor(-,\bb G_{m})_{\nis} $ in $\ShvA$ and for $n \geq 1$ inductively define $\Gmn{n+1}:= \Gmn{n} \otimes \Gmn{1}$.
		For all $n\geq 0$ and $U\in \Smk$ the canonical map
		$$
		\cc X(\Gmn{n}\times U)
		\longrightarrow \inthom{\Delta^{op}\ShvA}{\Gmn{1}}{\mathcal{X}(\Gmn{n+1}\times U)}
		$$ is a local equivalence.
		
		\item (${\bb A}^{1}$-invariance)
		For all $U\in \Smk$ the canonical map
		$
		\cc X(U\times\bb A^1)\longrightarrow\cc X(U)
		$ is a local equivalence.
		
		\item (Nisnevich excision)
		For every elementary Nisnevich square in $\Smk$
		$$
		\xymatrix{
			U'\ar[r]\ar[d]&V'\ar[d]\\
			U\ar[r]&V }
		$$
		the induced square
		$$
		\xymatrix{
			\cc X(U')\ar[r]\ar[d]&\cc X(V')\ar[d]\\
			\cc X(U)\ar[r]&\cc X(V)
		}
		$$ is homotopy cartesian in the local model structure on $\Delta^{op}\ShvA$.
		
	\end{enumerate}
\end{dfn}

For $n\geq 0$ and every finitely generated field extension $K/k$,
we have the standard algebraic $n$-simplex
$$
\Delta^n_K
=
\spec(K[x_0,\ldots,x_n]/(x_0+\cdots+x_n-1)).
$$
For every $0\leq i\leq n$ we define a closed subscheme $v_i$ of $\Delta^n_K$ by the equations $x_{j}=0$ for $j\neq i$.
We write $\wh{\Delta}^{n}_{K/k}$ for the semilocalization of the standard algebraic $n$-simplex $\Delta^n_K$
with closed points the vertices $v_{0},\dots,v_{n}\in \Delta^n_K$.

\begin{dfn}\label{suslindef}
    Similarly to \cite[Axioms 1.1]{garkusha2019framed}, we say that $\mathcal{X}$ is \textit{very effective} or \textit{satisfies Suslin's contractibility} 
    if for every $U \in \Sm$ and every finitely generated field extension $K/k$ the diagonal of the bisimplicial abelian group
    $\mathcal{X}(\Gmn{1} \times U)(\semilocalsimplex) $
    is contractible.
\end{dfn}

Since we assume that $p$ is invertible in $\Cor$ the following lemma holds.
\begin{lem}\label{invertplemma}
	If $F : \Cor \rightarrow \Ab$ is an additive functor, then $F$ factors over the full subcategory of $\bb Z[1/p]$-modules $\mathrm{Mod}_{\bb Z[1/p]} \subseteq \Ab$.
	In particular the inclusion functor $\mathrm{Mod}_{\bb Z[1/p]} \rightarrow \Ab$ induces an equivalence of categories
	$\Shv(\Cor,\Ab) \simeq \Shv(\Cor,\mathrm{Mod}_{\bb Z[1/p]}).$
\end{lem}
\begin{proof}
	If $F : \Cor \rightarrow \Ab$ is additive, then $F$ is an $\Ab$-enriched functor. By the $\Ab$-enriched co-Yoneda lemma we can write $F$ as the following coend: 
	for all $U \in \Smk$ we have an isomorphism in $\Ab$,
	$$F(U) \cong \overset{X \in \Smk}{\int} F(X) \otimes \Cor(U,X). $$
	Since $p$ is invertible in $\Cor(U,X)$ we have a canonical isomorphism $\Cor(U,X) \cong \Cor(U,X) \otimes \bb Z[1/p]$.
	Since the functor $-\otimes \bb Z[1/p] : \Ab \rightarrow \Ab$ is a left adjoint, it preserves coends, so we can compute
	$$F(U) \cong \overset{X \in \Smk}{\int} F(X) \otimes \Cor(U,X) \cong \overset{X \in \Smk}{\int} F(X) \otimes (\Cor(U,X) \otimes \bb Z[1/p]) \cong $$
	$$  \bb Z[1/p] \otimes \overset{X \in \Smk}{\int} F(X) \otimes \Cor(U,X) \cong \bb Z[1/p] \otimes F(U)$$
	which shows that $F(U)$ is a $\bb Z[1/p]$-module.
\end{proof}

For some of our results we will also have to make additional assumptions on the category of correspondences $\Cor$.

\begin{dfn}Let $\mathrm{Fr}_*(k)$ be the category of Voevodksy's framed correspondences (see \cite[Definition 2.3]{garkusha2021framedmotives}).
		For each $V \in \Smk$ let $\sigma_V : V \rightarrow V$ be the level 1 
		explicit framed correspondence $(\{0\}\times V, \bb A^1 \times V, \mathrm{pr}_{\bb A^1}, \mathrm{pr}_V )$.	\begin{enumerate} \label{framedCorDef}
	
		\item 
		We say that the category of correspondences $\Cor$ \textit{has framed correspondences} if there is a functor $\Phi: \mathrm{Fr}_*(k) \rightarrow \Cor$ which is the identity on objects and which takes every $\sigma_V$ to the identity of $V$. 
		
		\item We say that $\Cor$ \textit{satisfies the $\widehat{\Delta}$-property}
		 if for every  $ n > 0$ and for every finitely generated field extension $K/k$ the diagonal of $M_{\Cor}(\Gmn{n})(\semilocalsimplex)$ is quasi-isomorphic to $0$.
		 Here  $M_{\Cor} : \Sm \rightarrow \Ch(\ShvA)$ is the enriched motive functor
		 $M_{\Cor}(U):= C_*\Cor(-,U)_{\nis}.$
	\end{enumerate}
Basic examples satisfying both items are given by the categories of finite
correspondences $Cor$ or Milnor--Witt correspondences $\widetilde{Cor}$.
\end{dfn}

\section{The local model structure}\label{sectionmodelstructures}

In \cite[Section 3]{bonart2022paper1} we constructed a model structure on $\Ch(\ShvA)$ that is cellular, strongly left proper, weakly finitely generated, monoidal and satisfies the monoid axiom.
In this section we construct a model structure on $\Ch_{\geq 0}(\ShvA)$ that is cellular, strongly left proper, weakly finitely generated, monoidal, satisfies the monoid axiom, and in which weak equivalences are local quasi-isomorphisms. We construct the model structure by taking the right transferred model structure along the inclusion $\Ch_{\geq 0}(\ShvA) \rightarrow \Ch(\ShvA)$.
We then transfer the model structure along the Dold-Kan correspondence, to get a model structure on $\Delta^{op}\ShvA$ that is cellular, strongly left proper, weakly finitely generated, monoidal, satisfies the monoid axiom, and in which weak equivalences are stalkwise weak equivalences of simplicial sets.

Let us now start by constructing the model structure on $\Ch_{\geq 0}(\ShvA)$.
We have an inclusion functor $\iota : \Ch_{\geq 0}(\ShvA) \rightarrow \Ch(\ShvA)$. The inclusion functor $\iota$ has a left adjoint $\tau_{\text{naive}} : \Ch(\ShvA) \rightarrow \Ch_{\geq 0}(\ShvA)$, called the naive truncation functor. It sends
$\dots \rightarrow A_1 \rightarrow A_0 \rightarrow A_{-1} \rightarrow \dots $
to 
$\dots \rightarrow A_1 \rightarrow A_0.$
The inclusion functor $\iota$ also has a right adjoint $\tau_{\text{good}}$, called the good truncation functor. It sends 
$$\dots \rightarrow A_1 \rightarrow A_0 \overset{\partial^0_A}{\rightarrow} A_{-1} \rightarrow \dots $$
to 
$\dots \rightarrow A_1 \rightarrow \ker(\partial^0_A).$
So we have $\tau_{\text{naive}} \dashv \iota \dashv \tau_{\text{good}}.$

\begin{lem}\label{iotaLcofib}
	The endofunctor $\iota\tau_{\text{naive}} : \Ch(\ShvA) \rightarrow \Ch(\ShvA)$ preserves cofibrations.
\end{lem}
\begin{proof}
	Since $\iota\tau_{\text{naive}}$ is a left adjoint functor, it suffices to check it on the set of generating cofibrations $$I_{\Ch(\ShvA)} = \{ \Cor(-,X)_{\nis} \otimes S^n\bb Z \rightarrow \Cor(-,X)_{\nis} \otimes D^n\bb Z  \mid n \in \bb Z, X \in \Smk  \}.$$
	So take $n \in \bb Z$, $X \in \Smk$ and consider the map $$f : \Cor(-,X)_{\nis} \otimes S^n\bb Z \rightarrow \Cor(-,X)_{\nis} \otimes D^n\bb Z. $$
	If $n \geq 0$ then $\iota\tau_{\text{naive}}(f) = f$ is a cofibration.
	If $n \leq -2$ then $\iota\tau_{\text{naive}}(f) = 0$ is a cofibration.
	If $n = -1$ then $\iota\tau_{\text{naive}}(f)$ is the map
	$0 \rightarrow \Cor(-,X)_{\nis} \otimes S^0\bb Z$
	which is a cofibration, due to the following pushout square
	$$\xymatrix{ \Cor(-,X)_{\nis} \otimes S^{-1}\bb Z  \ar[r] \ar[d] & \Cor(-,X)_{\nis} \otimes D^{-1}\bb Z \ar[d] \\
	0 \ar[r] & \Cor(-,X)_{\nis} \otimes S^0\bb Z } $$
as required.
\end{proof}

\begin{dfn}
	Given a model category $M$ and an adjunction $L : N \rightleftarrows M : R$, we say that the right transferred model structure along the adjunction $L \dashv R$ exists, if there exists a model structure on $N$, such that a morphism $f$ is a weak equivalence (respectively cofibration) in $N$ if and only if $L(f)$ is a weak equivalence (respectively cofibration) in $M$. 
\end{dfn}

\begin{lem}\label{lemmacofibgen}
	The left transferred model structure on $\Ch_{\geq 0}(\ShvA)$ along the adjunction 
	$$ \iota : \Ch_{\geq 0}(\ShvA) \rightleftarrows \Ch(\ShvA) : \tau_{\textrm{good}}  $$
	exists. The resulting model structure on $\Ch_{\geq 0}(\ShvA)$ is cofibrantly generated.
\end{lem}
\begin{proof}
	We use \cite[Theorem 2.23]{bayeh2015lefttransfer}. All involved categories are locally presentable, and $\Ch(\ShvA)$ is cofibrantly generated, so the theorem is applicable.
	We now have to show that $$\mathrm{RLP}( \iota^{-1}(\{ \text{cofibrations} \}) ) \subseteq \iota^{-1}(\{ \text{weak equivalences} \}).$$
	
	So take $p : X \rightarrow Y$ with $p \in \mathrm{RLP}( \iota^{-1}(\{ \text{cofibrations} \}) )$. We want to show that $\iota(p)$ is a weak equivalence in $\Ch(\ShvA)$. We will show that $\iota(p)$ is a trivial fibration, by showing that it has the right lifting property with respect to cofibrations.
	Let $f : A \rightarrow B$ be a cofibration in $\Ch(\ShvA)$ and consider a lifting problem
	$$\xymatrix{A \ar[d]_f \ar[r] & \iota X \ar[d]^{\iota p} \\
	B \ar[r] \ar@{-->}[ur] & \iota Y  & .}$$
By adjunction this diagram has a lift, if and only if the following diagram has a lift
$$\xymatrix{\tau_{\text{naive}}A \ar[d]_{\tau_{\text{naive}}f} \ar[r] & X \ar[d]^{p} \\
	\tau_{\text{naive}}B \ar[r] \ar@{-->}[ur] & Y  & .}$$
Since $p \in \mathrm{RLP}( \iota^{-1}(\{ \text{cofibrations} \}) )$, one has to show that $\tau_{\text{naive}}f \in \iota^{-1}(\{ \text{cofibrations} \})$. 
One has to show that $\iota\tau_{\text{naive}}f$ is a cofibration. As $f$ is a cofibration, this follows from Lemma~\ref{iotaLcofib}.
\end{proof}

We now have a model structure on $\Ch_{\geq 0}(\ShvA)$, in which a morphism $f$ is weak equivalence (respectively cofibration) if and only if $\iota f$ is a weak equivalence (respectively cofibration) in $\Ch(\ShvA)$, and a morphism is a fibration in $\Ch_{\geq 0}(\ShvA)$ if and only if it has the right lifting property with respect to all trivial cofibrations. Furthermore, the adjunction $$\iota : \Ch_{\geq 0}(\ShvA) \rightleftarrows \Ch(\ShvA) : \tau_{\text{good}}$$
is a Quillen adjunction.
Since weak equivalences in $\Ch(\ShvA)$ are the local quasi-isomorphisms, 
it follows that also weak equivalences in $\Ch_{\geq 0}(\ShvA)$ are the local quasi-isomorphisms.

\begin{lem}
	$\Ch_{\geq 0}(\ShvA)$ is a monoidal model category.
\end{lem}
\begin{proof}
	Let us verify the pushout product axiom.
	Let $f, g$ be two cofibrations in $\Ch_{\geq 0}(\ShvA)$, and let $f \square g$ be their pushout-product.
	Since $\iota: \Ch_{\geq 0}(\ShvA) \rightarrow \Ch(\ShvA)$ is a strong monoidal left adjoint functor, we have an isomorphism of arrows
	$\iota(f\square g) \cong \iota(f) \square \iota(g).$
	As $f, g$ are cofibrations in $\Ch_{\geq 0}(\ShvA)$, we see that $\iota(f), \iota(g)$ are cofibrations in $\Ch(\ShvA)$. Since $\Ch(\ShvA)$ is a monoidal model category, $\iota(f) \square \iota(g)$ is a cofibration in $\Ch(\ShvA)$. So $f\square g$ is a cofibration in $\Ch_{\geq 0}(\ShvA)$.
	Also, if $f$ or $g$ is a trivial cofibration in $\Ch_{\geq 0}(\ShvA)$, then $\iota(f)$ or $\iota(g)$ is a trivial cofibration in $\Ch(\ShvA)$. Thus $\iota(f) \square \iota(g)$ is a trivial cofibration, hence $f\square g$ is a trivial cofibration. Therefore $\Ch_{\geq 0}(\ShvA)$ satisfies the pushout-product axiom.
	
	Let us verify the unit axiom.
	If $\mathbbm{1}_{\geq 0}$ is the monoidal unit of $\Ch_{\geq 0}(\ShvA)$, and $\mathbbm{1}$ is the monoidal unit of $\Ch(\ShvA)$, then since $\iota$ is strong monoidal we have an isomorphism
	$ \iota \mathbbm{1}_{\geq 0} \cong \mathbbm{1}.$
	As $\mathbbm{1} = \Cor(-,pt)_{\nis}$ is cofibrant in $\Ch(\ShvA)$ it follows that $\mathbbm{1}_{\geq 0}$ is cofibrant in  $\Ch_{\geq 0}(\ShvA)$. This implies the unit axiom.
\end{proof}

\begin{lem}
	$\Ch_{\geq 0}(\ShvA)$ satisfies the monoid axiom.
\end{lem}
\begin{proof}
	Let $W_{\geq 0}$ denote the class of weak equivalences and $CW_{\geq 0}$ denote the class of trivial cofibrations in $\Ch_{\geq 0}(\ShvA)$.
	Let $W$ denote the class of weak equivalences and $CW$ denote the class of trivial cofibrations in $\Ch(\ShvA)$.
	We need to show that
	$$((CW_{\geq 0})\otimes \Ch_{\geq 0}(\ShvA))-\cof \subseteq W_{\geq 0}.$$
	Since $W_{\geq 0} = \iota^{-1}(W)$, this means we have to show that
	$$\iota(((CW_{\geq 0})\otimes \Ch_{\geq 0}(\ShvA))-\cof) \subseteq W.$$
	
	Since $\iota$ is a strong monoidal left adjoint functor we have
	$$\iota(((CW_{\geq 0})\otimes \Ch_{\geq 0}(\ShvA))-\cof) \subseteq (\iota(CW_{\geq 0}) \otimes \Ch(\ShvA) )-\cof.$$
	Since $\iota$ preserves trivial cofibrations we have $\iota(CW_{\geq 0}) \subseteq CW$.
	Since $\Ch(\ShvA)$ satisfies the monoid axiom (see~\cite{bonart2022paper1}), it follows that
	$$(\iota(CW_{\geq 0}) \otimes \Ch(\ShvA) )-\cof \subseteq (CW \otimes \Ch(\ShvA) )-\cof \subseteq W.$$
	Hence $\Ch_{\geq 0}(\ShvA)$ satisfies the monoid axiom.
\end{proof}

\begin{lem}\label{lemmaIgeq0}
	Let $I_{\Ch(\ShvA)}$ be a set of generating cofibrations for $\Ch(\ShvA)$. Then 
	the set $\tau_{\text{naive}}(I_{\Ch(\ShvA)})$ is a set of generating cofibrations of $\Ch_{\geq 0}(\ShvA)$.
	In particular, the model category $\Ch_{\geq 0}(\ShvA)$ has a set of generating cofibrations with finitely presented domains and codomains.
\end{lem}
\begin{proof}
	By Lemma \ref{iotaLcofib} all morphisms from $\tau_{\text{naive}}(I_{\Ch(\ShvA)})$ are cofibrations in $\Ch_{\geq 0}(\ShvA)$.
	Let $f$ be a cofibration in $\Ch_{\geq 0}(\ShvA)$. We claim that $f \in (\tau_{\text{naive}}(I_{\Ch(\ShvA)}))-\cof.$
	Since $f$ is a cofibration in $\Ch_{\geq 0}(\ShvA)$, also $\iota f$ is a cofibration in $\Ch(\ShvA)$. Since $I_{\Ch(\ShvA)}$ is a set of generating cofibrations for $\Ch(\ShvA)$, it follows that $\iota f \in I_{\Ch(\ShvA)}-\cof.$
	But then $$f \cong \tau_{\text{naive}}\iota f \in \tau_{\text{naive}}(I_{\Ch(\ShvA)}-\cof) \subseteq (\tau_{\text{naive}}(I_{\Ch(\ShvA)}))-\cof.$$
	Therefore $\tau_{\text{naive}}(I_{\Ch(\ShvA)})$ is a set of generating cofibrations for $\Ch_{\geq 0}(\ShvA)$.
	
	Since the set
	$$\{\Cor(-,X)_{\nis} \otimes S^n \bb Z \rightarrow \Cor(-,X)_{\nis} \otimes D^n \bb Z \mid X \in \Smk, n \in \bb Z \}$$
	is a set of generating cofibrations with finitely presented domains and codomains for $\Ch(\ShvA)$, it follows that $$\{\Cor(-,X)_{\nis} \otimes S^n \bb Z \rightarrow \Cor(-,X)_{\nis} \otimes D^n \bb Z \mid X \in \Smk, n \geq 0 \} \cup \{ 0 \rightarrow \Cor(-,X)_{\nis} \otimes S^0\bb Z \mid X \in \Smk \}$$
	is a set of generating cofibrations with finitely presented domains and codomains of $\Ch_{\geq 0}(\ShvA)$.
\end{proof}

Next, we want to show that $\Ch_{\geq 0}(\ShvA)$ is weakly finitely generated. To this end, 
we need to define a set of weakly generating trivial cofibrations $J^\prime$.
For this we need to construct a certain set of morphisms similar to \cite[Definition 3.3]{bonart2022paper1}.
\begin{dfn}
	For every elementary Nisnevich square $Q \in \cc Q$ of the form $$\xymatrix{ U^\prime \ar[r]^\beta \ar[d]^\alpha & X^\prime \ar[d]^\gamma \\
		U \ar[r]^\delta & X} $$
	we have a square 
	$$ \xymatrix{ \Cor(-,U^\prime)_{\nis} \ar[r]^{\beta_*} \ar[d]^{\alpha_*} & \Cor(-,X^\prime)_{\nis} \ar[d]^{\gamma_*} \\
		\Cor(-,U)_{\nis} \ar[r]^{\delta_*}  & \Cor(-,X)_{\nis} } $$
	in $\Ch(\ShvA)$.
	Take the homological mapping cyinder $C$ of the map $\Cor(-,U^\prime)_{\nis} \rightarrow \Cor(-,X^\prime)_{\nis}$, so that the map factors as  $\xymatrix{\Cor(-,U^\prime)_{\nis} \: \ar[r]& C \ar[r] & \Cor(-,X^\prime)_{\nis} }$. Define an object $s_Q := \Cor(-,U)_{\nis} \underset{\Cor(-,U^{\prime})_{\nis}}{\coprod} C$.
	Next take the homological mapping cylinder $t_Q$ of the map $s_Q = \Cor(-,U)_{\nis} \underset{\Cor(-,U^{\prime})_{\nis}}{\coprod} C \rightarrow \Cor(-,X)_{\nis}$, so that it factors as $\xymatrix{s_Q \: \ar[r]^{p_Q}& t_Q \ar[r] & \Cor(-,X)_{\nis} }$.
	The map $p_Q : s_Q \rightarrow t_Q$ is a trivial cofibration between finitely presented objects of $\Ch_{\geq 0}(\ShvA)$.
	
	Let $\mathcal{Q}$ be the set of all elementary Nisnevich squares. Define
	a set of morphisms $J_{\mathcal{Q}} := \{ p_Q \mid Q \in \mathcal{Q}  \}.$
	Let $I_{\Ch_{\geq 0}(\Ab)}$ be a set of generating cofibrations with finitely presented domains and codomains for Quillen's standard projective model structure on $\Ch(\Ab)_{\geq 0}$. We define sets of morphisms in $\Ch_{\geq 0}(\ShvA)$
	$$J_{\proj} := \{ 0 \rightarrow \Cor(-,X)_{\nis} \otimes D^n\bb Z \mid X \in \Smk, n \geq 0 \}$$
	and
	$$J^\prime:=J_{\proj} \cup (J_{\mathcal{Q}} \square I_{\Ch_{\geq 0}(\Ab)}),$$
	where $J_{\mathcal{Q}} \square I_{\Ch_{\geq 0}(\Ab)}$ is the set of all morphisms which are a pushout product of a morphisms from $J_{\mathcal{Q}}$ and $I_{\Ch(\Ab)_{\geq 0}}$.
\end{dfn}
Note that all morphisms from $I_{\Ch_{\geq 0}(\Ab)}$ are cofibrations and all morphisms from $J_{\proj}$ and $J_{\mathcal{Q}}$ are trivial cofibrations. Since $\Ch_{\geq 0}(\ShvA)$ is a monoidal model category it follows that all morphisms from $J^\prime$ are trivial cofibrations.

\begin{lem}\label{lemmaJproj}
	A morphism $f : A \rightarrow B$ in $\Ch_{\geq 0}(\ShvA)$ has the right lifting property with respect to $J_{\proj}$ if and only if for every $n \geq 1$ the map $f_n : A_n \rightarrow B_n$ is sectionwise surjective.
\end{lem}

\begin{proof}
	For every $n \geq 0, X \in \Smk$ we can solve the lifting problem
	$$\xymatrix{0 \ar[d] \ar[r] & A \ar[d]^f \\
	            \Cor(-,X)_{\nis} \otimes D^n\bb Z \ar[r] \ar@{..>}[ur] & B} $$
            in $\Ch_{\geq 0}(\ShvA)$ if and only if $f_{n+1}: A(X)_{n+1} \rightarrow B(X)_{n+1}$ is surjective in $\Ab$.
\end{proof}

\begin{lem}\label{lemmaifibrant}
	For an object $A$ in $\Ch_{\geq 0}(\ShvA)$ the following are equivalent:
	\begin{enumerate}
		\item $\iota(A)$ is fibrant in $\Ch(\ShvA)$.
		\item $A$ is fibrant in $\Ch_{\geq 0}(\ShvA)$.
		\item $A \rightarrow 0$ has the right lifting property with respect to $J_{\mathcal{Q}} \square I_{\Ch_{\geq 0}(\Ab)}$.
	\end{enumerate}
\end{lem}
\begin{proof}
$(1)\Longrightarrow(2)$.
	If $\iota(A)$ is fibrant in $\Ch(\ShvA)$, then $A \cong \tau_{\text{good}}(\iota(A))$ is fibrant in $\Ch_{\geq 0}(\ShvA)$ because $\tau_{\text{good}}$ is a right Quillen functor.
	
$(2)\Longrightarrow(3)$.	If $A$ is fibrant in $\Ch_{\geq 0}(\ShvA)$, then $A \rightarrow 0$ has the right lifting property with respect to all trivial cofibrations, hence it has the right lifting property with respect to $J_{\mathcal{Q}} \square I_{\Ch_{\geq 0}(\Ab)}$. 
	
$(3)\Longrightarrow(1)$. Assume that $A \rightarrow 0$ has the right lifting property with respect to $J_{\mathcal{Q}} \square I_{\Ch_{\geq 0}(\Ab)}$. We want to show that $\iota(A)$ is fibrant in $\Ch(\ShvA)$.
	By \cite[Lemma 3.4]{bonart2022paper1} we have to show that $A(\emptyset) \rightarrow 0$ is a quasi-isomorphism, and that $A$ sends elementary Nisnevich squares to homotopy pullback squares.
	Since $A$ is a chain complex of sheaves, we have $A(\emptyset) = 0$. Let us now show that $A$ sends elementary Nisnevich squares to homotopy pullback squares. Let $Q$ be an elementary Nisnevich square.
	For $X, Y \in \Ch_{\geq0}(\ShvA)$ let $\inthom{\Ch_{\geq 0}(\ShvA)}{X}{Y}$ be the internal hom of $\Ch_{\geq 0}(\ShvA)$ and let $$\mathrm{map}^{\Ch}(X,Y) \in \Ch_{\geq 0}(\Ab)$$ be defined by $$\mathrm{map}^{\Ch}(X,Y) := \inthom{\Ch_{\geq 0}(\ShvA)}{X}{Y}(pt).$$
	
	The square $A(Q)$ will be a homotopy pullback square in $\Ch(\Ab)$ if and only if the map
	$$p_Q^* :  \mathrm{map}^{\Ch}(t_Q,A) \rightarrow \mathrm{map}^{\Ch}(s_Q,A) $$
	is a quasi-isomorphism in $\Ch_{\geq 0}(\Ab)$.
	To show that $p_Q^*$ is a quasi-isomorphism, it suffices to show that $p_Q^*$ is a trivial fibration in $\Ch_{\geq 0}(\Ab)$. For that we need to show that  $p_Q^*$ has the right lifting property with respect to $I_{\Ch(\Ab)_{\geq 0}}$.  Now for every map $f : M \rightarrow N$ in $I_{\Ch(\Ab)_{\geq 0}}$ a square
	$$\xymatrix{M \ar[d]_f \ar[r] & \mathrm{map}^{\Ch}(t_Q,A) \ar[d]^{p_Q^*}  \\
		N \ar[r] \ar@{..>}[ur] & \mathrm{map}^{\Ch}(s_Q,A) }$$
	has a lift in $\Ch_{\geq 0}(\Ab)$ if and only if the square
	$$\xymatrix{ t_Q \otimes M \underset{s_Q \otimes M}{\coprod} s_Q \otimes N \ar[d]_{p_Q \square f} \ar[r] & A \ar[d]  \\
		t_Q \otimes N \ar[r] \ar@{..>}[ur] & 0 }$$
	has a lift in $\Ch_{\geq 0}(\ShvA)$. This lift exists, because $A \rightarrow 0$ has the right lifting property with respect to $J_{\mathcal{Q}} \square I_{\Ch(\Ab)_{\geq 0}}$.
\end{proof}

In what follows, let $\Ch(\PshA)_{\proj}$ be the model category $\Ch(\PshA)$ with standard projective model structure. Let $\Ch(\PshA)_{\nis}$ be the model category $\Ch(\PshA)$ with local projective model structure. See \cite[Section 3]{bonart2022paper1} for details.
	Let $L_{\nis} : \Ch(\PshA) \leftrightarrows \Ch(\Shv)A : U_{\nis}$ be the adjunction consisting of the sheafification and the forgetful functors.

\begin{prop}\label{propJprimegeq0}
	Let $f : A \rightarrow B$ be a morphism in $\Ch_{\geq 0}(\ShvA)$ such that $B$ is fibrant and $f$ has the right lifting property with respect to $J^\prime$.
	Then $f$ is a fibration in $\Ch_{\geq 0}(\ShvA)$.
\end{prop}
\begin{proof}	
	Our first claim is that $A$ is fibrant.
	Since $B$ is fibrant, by Lemma \ref{lemmaifibrant} $B \rightarrow 0$ has the right lifting property with respect to $J_{\mathcal{Q}} \square I_{\Ch_{\geq 0}(\Ab)}$. Since $f$ has the right lifting property with respect to $J_{\mathcal{Q}} \square I_{\Ch_{\geq 0}(\Ab)}$ it follows that $A \rightarrow 0$ has the right lifting property with resepct to $J_{\mathcal{Q}} \square I_{\Ch_{\geq 0}(\Ab)}$. Lemma \ref{lemmaifibrant} implies $A$ is fibrant.
	
	Next, let $D^{-1}B_0 \in \Ch(\ShvA)$ denote the chain complex
	$$\dots 0 \rightarrow 0 \rightarrow B_0 \overset{id}{\rightarrow} B_0 \rightarrow 0 \rightarrow \dots$$
	that is $B_0$ in degree $0$ and $-1$, and which is $0$ everywhere else.
	We claim that $D^{-1}B_0$ is fibrant in $\Ch(\ShvA)$. Indeed, the map $U_{\nis}D^{-1}B_0 \rightarrow 0$ is a trivial fibration in $\Ch(\PshA)_{\proj}$, hence it is also a trivial fibration in $\Ch(\PshA)_{\nis}$. Therefore $D^{-1}B_0 \rightarrow 0$ is a trivial fibration in $\Ch(\ShvA)$, and so
	$D^{-1}B_0$ is fibrant. Note that $\tau_{\text{good}}(D^{-1}B_0) = 0$ in $\Ch_{\geq 0}(\ShvA)$.
	
	In particular, $\iota(A) \oplus D^{-1}B_0$ is fibrant in $\Ch(\ShvA)$ and we have that
	$$\tau_{\text{good}}(\iota(A) \oplus D^{-1}B_0) \cong \tau_{\text{good}}(\iota(A)) \oplus \tau_{\text{good}}(D^{-1}B_0) \cong A \oplus 0 = A.$$
	
	Define $g : D^{-1}B_0 \rightarrow \iota(B)$ in $\Ch(\ShvA)$ as the map
	$$\xymatrix{ \dots \ar[r] \ar[d] & 0 \ar[r] \ar[d]& B_0 \ar[r]^{id} \ar[d]^{id}& B_0 \ar[r] \ar[d]& 0 \ar[r] \ar[d]& \dots \\
	\dots \ar[r] & B_1 \ar[r] & B_0 \ar[r] & 0 \ar[r] & 0 \ar[r] & \dots   }$$
	
	Then $\iota(f) + g : \iota(A) \oplus D^{-1}B_0 \rightarrow \iota(B)$ is a map between fibrant objects, and we have a commutative diagram where the horizontal maps are isomorphisms
	$$\xymatrix{ \tau_{\text{good}}(\iota(A) \oplus D^{-1}B_0) \ar[r]^(0.65)\sim \ar[d]_{\tau_{\text{good}}(\iota(f)+g)}  & A \oplus 0 \ar[d]^{f+0} \ar@{=}[r] & A \ar[d]^f \\ 
	\tau_{\text{good}}(\iota(B)) \ar[r]^\sim & B \ar@{=}[r] & B } $$
	We want to show that $f$ is a fibration in $\Ch_{\geq 0}(\ShvA)$.
	Since $\tau_{\text{good}}$ is a right Quillen functor, we now just need to show that $\iota(f) + g$ is a fibration in $\Ch(\ShvA)$.
	For this it suffices to show that $U_{\nis}(\iota(f) + g)$ is a fibration in $\Ch(\PshA)_{\nis}$.
	Since $U_{\nis}\iota(A \oplus D^{-1}B_0)$ and $U_{\nis}\iota(B)$ are fibrant in $\Ch(\PshA)_{\nis}$, it suffices by \cite[Proposition 3.3.16]{hirschhorn2003model} to show that $U_{\nis}(\iota(f) + g)$ is a fibration in $\Ch(\PshA)_{\proj}$.
	So we have to show that the map $\iota(f) + g$ is sectionwise an epimorphism in $\Ch(\Ab)$.
	In degree $n \geq 1$ the map $\iota(f) : \iota(A) \rightarrow \iota(B)$ is sectionwise surjective, because of Lemma \ref{lemmaJproj} and the fact that $f$ satisfies the right lifting property with respect to $J_{\proj}$.
	In degree $n\leq -1$ the map  $\iota(f) + g$ is sectionwise surjective, because $\iota(B)_n = 0$.
	Finally, in degree $n = 0$ the map $\iota(f) + g$ is sectionwise surjective, because $g : D^{-1}B_0 \rightarrow \iota(B)$ is sectionwise surjective in degree $0$.
	So $U_{\nis}(\iota(f) + g)$ is a fibration in $\Ch(\PshA)_{\proj}$. Then $\iota(f) + g$ is a fibration in $\Ch(\ShvA)$, and then $f \cong \tau_{\text{good}}(\iota(f) + g)$ is a fibration in $\Ch_{\geq 0}(\ShvA)$.
\end{proof}

\begin{cor}\label{corweakfingen}
	$\Ch_{\geq 0}(\ShvA)$ is weakly finitely generated and $J^\prime$ is a set of weakly generating trivial cofibrations for $\Ch_{\geq 0}(\ShvA)$.
\end{cor}
\begin{proof}
	By Lemma \ref{lemmacofibgen} $\Ch_{\geq 0}(\ShvA)$ is cofibrantly generated, so there exists a set $J$ of generating trivial cofibrations. Since every object in $\Ch_{\geq 0}(\ShvA)$ is small, the domains and codomains from $J$ are small.
	By Lemma \ref{lemmaIgeq0} $\Ch_{\geq 0}(\ShvA)$ has a set of generating cofibrations with finitely presented domains and codomains.
	All morphisms from $J^\prime$ are trivial cofibrations with finitely presented domains and codomains, so 
	Proposition \ref{propJprimegeq0} implies that $J^\prime$ is set of weakly generating trivial cofibrations for $\Ch_{\geq 0}(\ShvA)$.
\end{proof}

\begin{lem}
	The model category $\Ch_{\geq 0}(\ShvA)$ is cellular.
\end{lem}
\begin{proof}
	Due to Corollary \ref{corweakfingen} we just need to show that cofibrations in $\Ch_{\geq 0}(\ShvA)$ are effective monomorphisms.
	If $f$ is a cofibration in $\Ch_{\geq 0}(\ShvA)$, then $\iota(f)$ is a cofibration in $\Ch(\ShvA)$. Then $f$ is a monomorphism in $\Ch(\ShvA)$
	and in $\Ch_{\geq 0}(\ShvA)$.  Since $\Ch_{\geq 0}(\ShvA)$ is an abelian category, every monomorphism is effective. Hence $f$ is an effective monomorphism.
\end{proof}

\begin{lem}
	The model category 
	$\Ch_{\geq 0}(\ShvA)$ is strongly left proper in the sense of \cite[Definition 4.6]{dundas2003enriched}
\end{lem}
\begin{proof}
	If we have a pushout square
	$$\xymatrix{A \otimes Z \ar[r]^f \ar[d]_{g \otimes Z} & B \ar[d] \\
	C \otimes Z \ar[r]_h & D}$$
in $\Ch_{\geq 0}(\ShvA)$ with $f$ a weak equivalence and $g : A \rightarrow C$ a cofibration,
then the square
$$\xymatrix{\iota(A) \otimes \iota(Z) \ar[r]^\sim \ar[d]_{\iota(g) \otimes \iota(Z)} & \iota(A \otimes Z) \ar[r]^{\iota(f)} \ar[d]_{\iota(g \otimes Z)} & \iota(B) \ar[d] \\
\iota(C) \otimes \iota(Z) \ar[r]^\sim & \iota(C \otimes Z) \ar[r]_{\iota(h)} & \iota(D)  }$$
is a pushout square in $\Ch(\ShvA)$. Since $\iota(f)$ is a weak equivalence, $\iota(g)$ is a cofibration, and $\Ch(\ShvA)$ is strongly left proper
by~\cite{bonart2022paper1}, it follows that $\iota(h)$ is a weak equivalence in $\Ch(\ShvA)$. So $h$ is a weak equivalence in $\Ch_{\geq 0}(\ShvA)$.
\end{proof}

In summary, we have a model category $\Ch_{\geq 0}(\ShvA)$ that is cellular, weakly finitely generated and where the weak equivalences are the local quasi-isomorphisms. With respect to the usual tensor product of chain complexes $\otimes$ it is monoidal, strongly left proper and satisfies the monoid axiom.

We can transfer this model structure along the Dold-Kan correspondence
$$DK : \Ch_{\geq 0}(\ShvA) \overset{\sim}{\leftrightarrow} \Delta^{op}(\ShvA) : DK^{-1}.$$
So we define a model structure on $\Delta^{op}(\ShvA)$, where a morphism $f$ is a weak equivalence (respectively fibration, cofibration), if and only if $DK^{-1}(f)$ is a weak equivalence (respectively fibration, cofibration) in $\Ch_{\geq 0}(\ShvA)$.
Then weak equivalences in $\Delta^{op}\ShvA$ are the stalkwise weak equivalences of simplicial sets. Furthermore $\Delta^{op}\ShvA$ is weakly finitely generated and cellular.
From now on, weak equivalences in $\Delta^{op}\ShvA$ be called local equivalences, fibrations in $\Delta^{op}\ShvA$ will be called local fibrations, and fibrant objects in $\Delta^{op}\ShvA$ will be called locally fibrant objects.

Let $\otimes$ be the degreewise tensor product of $\Delta^{op}\ShvA$.
We want to show that $\Delta^{op}\ShvA$ is monoidal, strongly left proper and satisfies the monoid axiom with respect to $\otimes$.

The Dold-Kan correspondence is unfortunately not strongly monoidal with respect to the degreewise tensor product $\otimes$ on $\Delta^{op}\ShvA$ and the usual tensor product of chain complexes on $\Ch_{\geq 0}(\ShvA)$.
We define on $\Ch_{\geq 0}(\ShvA)$ the \textit{Dold-Kan twisted tensor product} $\otimesDK$ by
$$A \otimesDK B := DK^{-1}(DK(A) \otimes DK(B)). $$
Then the Dold-Kan correspondences is strongly monoidal with respect to the degreewise tensor product $\otimes$ on $\Delta^{op}\ShvA$ and the Dold-Kan twisted tensor product $\otimesDK$ on $\Ch_{\geq 0}(\ShvA)$.
So to show that $\Delta^{op}\ShvA$ is monoidal, strongly left proper and satisfies the monoid axiom with respect to $\otimes$, we now just need to show that $\Ch_{\geq 0}(\ShvA)$ is monoidal, strongly left proper and satisfies the monoid axiom with respect to $\otimesDK$.

\begin{lem}\label{lemfZmono}
	Let $f$ be a cofibration and $Z$ an object in $\Ch_{\geq 0}(\ShvA)$. Then $f \otimesDK Z$ is a monomorphism.
\end{lem}
\begin{proof}
	If $f : A \rightarrow B$ is a cofibration in $\Ch_{\geq 0}(\ShvA)$ then $f$ is a degreewise split monomorphism.
	The functor $DK: \Ch_{\geq 0}(\ShvA) \rightarrow \Delta^{op}\ShvA$ can be explicitly computed in degree $n \geq 0$ by
	$$DK(X)_n = \underset{\underset{\text{surjective}}{[n] \rightarrow [k]} }{\bigoplus} X_k .$$
	So $DK(f)$ is computed as the morphism
	$$DK(f)_n = \underset{\underset{\text{surjective}}{[n] \rightarrow [k]} }{\bigoplus} f_k : \underset{\underset{\text{surjective}}{[n] \rightarrow [k]} }{\bigoplus} A_k \rightarrow \underset{\underset{\text{surjective}}{[n] \rightarrow [k]} }{\bigoplus} B_k .$$
	This is a direct sum of split monomorphisms.
	So $DK(f)$ is a degreewise split monomorphism in $\Delta^{op}\ShvA$.
	Hence, if $Z$ is an object in $\Ch_{\geq 0}(\ShvA)$, then the degreewise tensor product
	$$DK(f) \otimes DK(Z) $$
	is again a split monomorphism in $\Delta^{op}\ShvA$.
	Since $DK^{-1}$ preserves monomorphisms, this then implies that
	$$f \otimesDK Z = DK^{-1}(DK(f) \otimes DK(Z))$$
	is a monomorphism in $\Ch_{\geq 0}(\ShvA)$.
\end{proof}

\begin{lem}
	$\Ch_{\geq 0}(\ShvA)$ satisfies the monoid axiom with respect to $\otimesDK$.
	So $\Delta^{op}\ShvA$ satisfies the monoid axiom with respect to $\otimes$.
\end{lem}
\begin{proof}
		Since $\ShvA$ is a Grothendieck category, we know that injective quasi-isomorphisms in $\Ch_{\geq 0}(\ShvA)$ are stable under pushouts and transfinite compositions. So to prove the monoid axiom we just need to show that for every trivial cofibration $f : A \rightarrow B$ in $\Ch_{\geq 0}(\ShvA)$ the morphism  $f \otimesDK Z$ is an injective quasi-isomorphism.
		By Lemma \ref{lemfZmono} we know that it is injective.
		So we just need to show that it is a weak equivalence.
		
		By \cite{nlab:eilenberg-zilber/alexander-whitney_deformation_retraction} we have for all $X, Y \in \Ch_{\geq 0}(\ShvA)$ a natural chain homotopy equivalence
		$$\nabla : X \otimes Y \rightarrow X \otimesDK Y$$
		between the usual tensor product of chain complexes and the Dold-Kan twisted tensor product.
		We then get a commutative diagram
		$$\xymatrix{ A \otimesDK Z \ar[r]^{f \otimesDK Z} & B \otimesDK Z\\
		             A \otimes Z \ar[r]^{f \otimes Z} \ar[u]^(0.45){\nabla} & B \otimes Z \ar[u]_(0.45){\nabla} } $$
	    where vertical maps are chain homotopy equivalences, and the lower horizontal map is a weak equivalence because $\Ch_{\geq 0}(\ShvA)$ satisfies the monoid axiom with respect to $\otimes$.
	    It follows that the upper horizontal map is a weak equivalence. So $\Ch_{\geq 0}(\ShvA)$ satisfies the monoid axiom with respect to $\otimesDK$.
\end{proof}

\begin{lem}
	$\Ch_{\geq 0}(\ShvA)$ is strongly left proper with respect to $\otimesDK$.
	So $\Delta^{op}\ShvA$ is strongly left proper with respect to $\otimes$.
\end{lem}
\begin{proof}
	Since $\ShvA$ is a Grothendieck category, quasi-isomorphisms in $\Ch_{\geq 0}(\ShvA)$ are stable under pushouts along monorphisms. For any cofibration $f$ the map $f \otimesDK Z$ is a monomorphism by Lemma \ref{lemfZmono}.  So $\Ch_{\geq 0}(\ShvA)$ is strongly left proper with respect to $\otimesDK$.
\end{proof}

\begin{lem}
	$\Ch_{\geq 0}(\ShvA)$ is a monoidal model category with respect to $\otimesDK$.
	So $\Delta^{op}\ShvA$ is a monoidal model category with respect to $\otimes$.
\end{lem}
\begin{proof}
	The unit for $\otimesDK$ is the chain complex $\bb Z$ concentrated in degree $0$. That is a cofibrant object, so $\Ch_{\geq 0}(\ShvA)$ satisfies the unit axiom.
	Let us now show the pushout-product axiom.
	The category of simplicial abelian groups $\Delta^{op}\Ab$ is monoidal and satisfies the monoid axiom with respect to the degreewise tensor product of chain complexes $\otimes$.
	If we define a Dold-Kan twisted tensor product $\otimesDK$ on chain complexes of abelian groups $\Ch_{\geq 0}(\Ab)$ by
	$$X \otimesDK Y = DK^{-1}(DK(X) \otimes DK(Y)) $$
	then $\Ch_{\geq 0}(\Ab)$ with the standard projective model structure and tensor product $\otimesDK$ is a monoidal model category satisfying the monoid axiom.
	Similarly, we can also define a Dold-Kan twisted tensor product $\otimesDK$ on chain complexes of presheaves $\Ch_{\geq 0}(\PshA)$, and it coincides with the Day convolution product induced by the Dold-Kan twisted tensor product on $\Ch_{\geq 0}(\Ab)$ and the monoidal structure of $\Cor$.
	By \cite[Theorem 5.5]{garkusha2019derived} it follows that $\Ch_{\geq 0}(\PshA)$ with standard projective model structure and the Dold-Kan twisted tensor product $\otimesDK$ is a monoidal model category.
	For $\Ch_{\geq 0}(\ShvA)$ the set $$\{\Cor(-,X)_{\nis} \otimes S^n \bb Z \rightarrow \Cor(-,X)_{\nis} \otimes D^n \bb Z \mid X \in \Smk, n \geq 0 \} \cup \{ 0 \rightarrow \Cor(-,X)_{\nis} \otimes S^0\bb Z \mid X \in \Smk \}$$ is a set of generating cofibrations. All these generating cofibrations are sheafifications of cofibrations from $\Ch_{\geq 0}(\PshA)$.
	So if $f$ and $g$ are generating cofibrations in $\Ch_{\geq 0}(\PshA)$, and $f \square g$ is the pushout-product with respect to $\otimesDK$, then we can find cofibrations $f^\prime$ and $g^\prime$ in $\Ch_{\geq 0}(\PshA)$ such that $f = L_{\nis}(f^\prime)$ and $g = L_{\nis}(g^\prime)$.
	Then $f \square g \cong L_{\nis}(f^\prime \square g^\prime)$,
	where the pushout-product $f^\prime \square g^\prime$ in $\Ch_{\geq 0}(\PshA)$ is taken with respect to $\otimesDK$. Since $\Ch_{\geq 0}(\PshA)$ is a monoidal model category with respect to $\otimesDK$ it follows that $f^\prime \square g^\prime$ is a cofibration in $\Ch_{\geq 0}(\PshA)$, and therefore $f \square g$ is a cofibration in $\Ch_{\geq 0}(\ShvA)$.
	All we need to show now is that a pushout-product of a cofibration with a trivial cofibration is a weak equivalence in $\Ch_{\geq 0}(\ShvA)$.
	So let $f : A \rightarrow B$ be a cofibration and $g : C \rightarrow D$ be a trivial cofibration in $\Ch_{\geq 0}(\ShvA)$. We need to show that the pushout-product $f \square g$ with respect to $\otimesDK$ is a weak equivalence in $\Ch_{\geq 0}(\ShvA)$.
	Consider the diagram
	$$\xymatrix{ A \otimesDK C \ar[d] \ar[r]^{A \otimesDK g} & A \otimesDK D \ar[d] \ar@/^2.0pc/[ddr] \\
		B \otimesDK C \ar[r]^(0.4)h \ar@/_1.5pc/[rrd]_{B \otimesDK g} & A \otimesDK D \underset{A \otimesDK C}{\coprod} B \otimesDK C \ar[dr]^(0.6){f \square g} \\
		& & B \otimesDK D } $$
	The morphism $h$ is a base change of $A \otimesDK g$. Since $g$ is a trivial cofibration and $\Ch_{\geq 0}(\ShvA)$ satisfies the monoid axiom with respect to $\otimesDK$, this means that $h$ is a weak equivalence in $\Ch_{\geq 0}(\ShvA)$.
	Similarly $B \otimesDK g$ is a weak equivalence in $\Ch_{\geq 0}(\ShvA)$. So by 2-of-3 it follows that $f \square g$ is a weak equivalence in  $\Ch_{\geq 0}(\ShvA)$. So $\Ch_{\geq 0}(\ShvA)$ is a monoidal model category.
\end{proof}

We document the above lemmas as follows.

\begin{prop}\label{propmodelcategory}
	The model category $\Delta^{op}\ShvA$ with the usual degreewise tensor product is cellular, weakly finitely generated, monoidal, strongly left proper and satisfies the monoid axiom.
\end{prop}

From now on, weak equivalences in $\Delta^{op}\ShvA$ be called local equivalences, fibrations in $\Delta^{op}\ShvA$ will be called local fibrations, and fibrant objects in $\Delta^{op}\ShvA$ will be called locally fibrant objects.

\section{Relation to $\Gamma$-spaces} \label{relatedtogamma}

For every natural number $n \geq 0$ let $n_+$ be the pointed set $\{0, \dots, n\}$ where $0$ is the basepoint. We write $\Gamma^{op}$ for the full subcategory of the category of pointed sets on the objects $n_+$. $\Gamma^{op}$ is equivalent to the category of finite pointed sets. We write $\Gamma$ for the opposite category of $\Gamma^{op}$. This category is equivalent to the category called $\Gamma$ in Segal's original paper \cite{S}.

In the additive context we do not need the category $\Gamma$ as a variable in contrast to framed motivic $\Gamma$-spaces in the sense of \cite{garkusha2019framed}. This section is to justify this fact (see Proposition \ref{generalgammaequiv}).
We also associate framed motivic $\Gamma$-spaces to enriched motivic $\Cor$-spaces (see Proposition \ref{framedgamma}).

Let $\mathcal{B}$ be an additive model category.
By $\GammaSpc(\mathcal{B})$ we denote the full subcategory of the functor category $\mathrm{Fun}(\Gamma^{op},\mathcal{B})$ consisting of those functors $\mathcal{X} : \Gamma^{op} \rightarrow B$ such that for every $n \in \bb N$ the canonical map $\mathcal{X}(n_+) \rightarrow \underset{i=1}{\overset{n}{\prod}} \mathcal{X}(1_+) $ is a weak equivalence in $\mathcal{B}$. This category is called the {\it category of special $\Gamma$-spaces in $\mathcal{B}$}.

We have a functor $\EM : \mathcal{B} \rightarrow \GammaSpc(\mathcal{B})$ given by the Eilenberg Maclane construction
$\EM(A)(n_+) := \underset{i=1}{\overset{n}{\bigoplus}} A.$
For a function $f : m_+ \rightarrow n_+$ between pointed finite sets, we define $$\EM(A)(f): \underset{j=1}{\overset{m}{\bigoplus}} A \rightarrow \underset{i=1}{\overset{n}{\bigoplus}} A$$
as follows. For $0 \leq i \leq n$ the $i$-th component $\EM(A)(f)_i : \underset{j=1}{\overset{m}{\bigoplus}} A \rightarrow A$ is $\EM(A)(f)_i := \sum_{j \in f^{-1}(\{i\})}\pi_j$, where $\pi_j : \underset{i=1}{\overset{m}{\bigoplus}} A \rightarrow A$ is the $j$-th projection morphism.

We have another functor $ev_1 : \GammaSpc(\mathcal{B}) \rightarrow \mathcal{B}$ given by $ev_1(\mathcal{X}) := \mathcal{X}(1_+). $

\begin{lem}
	The functor $ev_1 : \GammaSpc(\mathcal{B}) \rightarrow \mathcal{B}$ is left adjoint to $\EM : \mathcal{B} \rightarrow \GammaSpc(\mathcal{B})$.
\end{lem}
\begin{proof}
	Given a morphism $\phi: \mathcal{X}(1_+) \rightarrow A$ in $\mathcal{B}$, we get for every $n \in \bb N$ a morphism
	$$\mathcal{X}(n_+) \rightarrow  \underset{i=1}{\overset{n}{\bigoplus}}\mathcal{X}(1_+) \rightarrow\underset{i=1}{\overset{n}{\bigoplus}} A = \EM(A)(n_+),$$
	which together assemble into a morphism $\Phi(\phi) : \mathcal{X} \rightarrow \EM(A)$ in $\GammaSpc(\mathcal{B})$.
	Conversely, given a morphism $\psi : \mathcal{X} \rightarrow \EM(A)$ in $\GammaSpc(\mathcal{B})$, we can evaluate it at $1_+$ to get a morphism $$\Psi(\psi) : \mathcal{X}(1_+) \rightarrow \EM(A)(1_+) = A.$$
	It is obvious that for every $\phi :  \mathcal{X}(1_+) \rightarrow A$ we have $\Psi(\Phi(\phi)) = \phi$.
	Now take a morphism $\psi : \mathcal{X} \rightarrow \EM(A)$ in $\GammaSpc(\mathcal{B})$. We claim that $\Phi(\Psi(\psi)) = \psi$.
	Take $n \in \bb N$ and show that $\psi(n_+) :\mathcal{X}(n_+) \rightarrow \EM(A) = \underset{i=1}{\overset{n}{\bigoplus}} A$ is equal to 
	$\Phi(\Psi(\psi))(n_+) : \mathcal{X}(n_+) \rightarrow  \underset{i=1}{\overset{n}{\bigoplus}} \mathcal{X}(1_+) \rightarrow \underset{i=1}{\overset{n}{\bigoplus}} A$.
	By the universal property of the product $\bigoplus$ we need to take $i$ with $0 \leq i \leq n$ and show that the following diagram commutes
	$$\xymatrix{ \mathcal{X}(n_+) \ar[r]^{\psi(n_+)} \ar[d]^{\mathcal{X}(\pi_i)} &  \underset{i=1}{\overset{n}{\bigoplus}} A \ar[d]^{\pi_i} \\
		\mathcal{X}(1_+) \ar[r]^{\psi(1_+)} & A } $$
	But this just follows from the naturality of $\psi : \mathcal{X} \rightarrow \EM(A)$.
\end{proof}

\begin{dfn}
(1) Let $\mathcal{B}$ be an additive model category.
		A morphism $f : \mathcal{X} \rightarrow \mathcal{Y}$ in $\GammaSpc(\mathcal{B})$ is called a \textit{weak equivalence} if and only if for every $n \in \bb N$ the map $f(n_+) : \mathcal{X}(n_+) \rightarrow \mathcal{Y}(n_+)$ is a weak equivalence  in the model category $\mathcal{B}$.
		We write $W$ for the class of weak equivalences in $\GammaSpc(\mathcal{B})$.
		
		(2) We write $\mathrm{Ho}(\GammaSpc(\mathcal{B}))$ for the localization of $\GammaSpc(\mathcal{B})$ with respect to the class of weak equivalences $W$:
		$ \mathrm{Ho}(\GammaSpc(\mathcal{B})) :=\GammaSpc(\mathcal{B})[W^{-1}].$
\end{dfn}

\begin{rem}
(1)		All isomorphisms in $\GammaSpc(\mathcal{B})$ are weak equivalences. Weak equivalences in $\GammaSpc(\mathcal{B})$ satisfy the 2-out-of-3 property.
		
		(2) The functors $\EM : \mathcal{B} \rightarrow \GammaSpc(\mathcal{B})$ and $ev_1 : \GammaSpc(\mathcal{B}) \rightarrow \mathcal{B}$ preserve all weak equivalences.
		
		(3) It is a priori not obvious that the hom-sets of the category $\mathrm{Ho}(\GammaSpc(\mathcal{B}))$ are small.
		However, Proposition \ref{generalgammaequiv} below implies that they are in fact small.
\end{rem}

\begin{lem} \label{gammaweqiff}
	A morphism $\phi : ev_1(\mathcal{X}) \rightarrow A$ is a weak equivalence in $\mathcal{B}$ if and only if its adjoint morphism $\Phi(\phi) : \mathcal{X} \rightarrow \EM(A)$ is a weak equivalence in $\GammaSpc(\mathcal{B})$.
\end{lem}
\begin{proof}
	Let $\phi : ev_1(\mathcal{X}) \rightarrow A$ be a weak equivalence. Take $n \in \bb N$. Then $\Phi(\phi)$ evaluated at $n_+$ is defined as the composite
	$$\mathcal{X}(n_+) \rightarrow  \underset{i=1}{\overset{n}{\bigoplus}}\mathcal{X}(1_+) \rightarrow\underset{i=1}{\overset{n}{\bigoplus}} A = EM(A)(n_+). $$
	The first map is a weak equivalence, because $\mathcal{X}$ is a special $\Gamma$-space.
	The second map is a weak equivalence, because $\phi : \mathcal{X}(1_+) \rightarrow A$ is a weak equivalence.
	Therefore $\Phi(\phi) : \mathcal{X} \rightarrow \EM(A)$ is a weak equivalence.
	
	Conversely, let $\phi :  ev_1(\mathcal{X}) \rightarrow A$ be a map such that $\Phi(\phi)$ is a weak equivalence in $\GammaSpc(\mathcal{B})$. Then $\phi = \Phi(\phi)(1_+)$ is also a weak equivalence.
\end{proof}

The following lemma is folklore.

\begin{lem}\label{lemmaHomotopyTrafo}
	Let $\cc C, \cc D$ be categories, each equipped with a class of morphisms, called the weak equivalences, satisfying the $2$-out-of-$3$-property.
	Let $\mathrm{Ho}(\cc C)$, $\mathrm{Ho}(\cc D)$ be the homotopy categories of $\cc C$, $\cc D$, i.e. the categories obtained by inverting the weak equivalences. Let $\ell_{\cc C} : \cc C \rightarrow \mathrm{Ho}(\cc C)$ be the localization functor of $\cc C$, and $\ell_{\cc D} : \cc D \rightarrow \mathrm{Ho}(\cc D)$ be the localization functor of $\cc D$.
	Let $F, G : \cc C \rightarrow \cc D$ be functors sending weak equivalences in $\cc C$ to weak equivalences in $\cc D$.
	Let $\tau : F \rightarrow G$ be a natural transformation.
	Then the functors $F, G$ induce functors $\mathrm{Ho}(F), \mathrm{Ho}(G) : \mathrm{Ho}(\cc C) \rightarrow \mathrm{Ho}(\cc D)$ satisfying $\mathrm{Ho}(F) \circ \ell_{\cc C} = \ell_{\cc D} \circ F$, $\mathrm{Ho}(G) \circ \ell_{\cc C} = \ell_{\cc D} \circ G$, and $\tau : F \rightarrow G$ induces a natural transformation $\mathrm{Ho}(\tau): \mathrm{Ho}(F) \rightarrow \mathrm{Ho}(G)$ such that for every $A \in \cc C$, the component of $\mathrm{Ho}(\tau)$ at $A$ is given by $\mathrm{Ho}(\tau)_{A} = \ell_{\cc D}(\tau_A)$.
\end{lem}

The following statement informally says that $\Gamma$-spaces in an additive category $\mathcal{B}$ are entirely recovered by $\mathcal{B}$ itself (up to homotopy).

\begin{prop} \label{generalgammaequiv}
	The adjunction $ev_1 \dashv \EM$ induces an equivalence of categories
	$$\mathrm{Ho}(ev_1): \mathrm{Ho}(\GammaSpc(\mathcal{B})) \overset{\sim}{\leftrightarrows} \mathrm{Ho}(\mathcal{B}) : \mathrm{Ho}(\EM). $$
\end{prop}
\begin{proof}
	Since $ev_1$ and $\EM$ preserve weak equivalences, they induce two functors $\mathrm{Ho}(ev_1): \mathrm{Ho}(\GammaSpc(\mathcal{B})) \rightarrow \mathrm{Ho}(\mathcal{B})$ and $\mathrm{Ho}(\EM): \mathrm{Ho}(\mathcal{B}) \rightarrow \mathrm{Ho}(\GammaSpc(\mathcal{B}))  $ on the homotopy categories.
	For the adjunction $\ev_1 \dashv \EM$ there is a unit $\eta : \mathrm{Id_{\GammaSpc(\mathcal{B})}} \rightarrow \EM\circ ev_1$.
	By Lemma \ref{lemmaHomotopyTrafo}, applied to $F = \mathrm{Id}_{\GammaSpc(\cc B)}$, $G = \EM \circ ev_1$ and $\tau= \eta$ , it induces a natural transformation $\mathrm{Ho}(\eta) :  \mathrm{Id_{\mathrm{Ho}(\GammaSpc(\mathcal{B}))}} \rightarrow \mathrm{Ho}(\EM) \circ \mathrm{Ho}(ev_1) $.
	
	For every $\cc X \in \GammaSpc(\cc B)$ the identity morphism $ev_1(\cc X) \rightarrow ev_1(\cc X)$ is a weak equivalence, so by Lemma \ref{gammaweqiff} applied to $A = ev_1(\cc X)$, the adjunction unit map $\eta_{\cc X} : \cc X \rightarrow \EM(ev_1(\cc X))$ is a weak equivalence.
	This implies that the natural transformation $\mathrm{Ho}(\eta)$ is in fact a natural isomorphism of functors.
	
	Furthermore we have a strict equality $ev_1 \circ \EM = \mathrm{Id_{\mathcal{B}}} $, which implies that $ \mathrm{Ho}(ev_1) \circ \mathrm{Ho}(\EM) = \mathrm{Id_{\mathrm{Ho}(\mathcal{B})}} $.
	So $\mathrm{Ho}(ev_1)$ is an equivalence with pseudo-inverse $\mathrm{Ho}(\EM)$.
\end{proof}

Let $\mathrm{Fr}_*(k)$ be the category of framed correspondences.
For each $V \in \Smk$ let $\sigma_V : V \rightarrow V$ be the level 1 explicit framed correspondence $(\{0\}\times V, \bb A^1 \times V, \mathrm{pr}_{\bb A^1}, \mathrm{pr}_V )$.
For the next result, assume that $\Cor$ has framed correspondences in the sense of Definition \ref{framedCorDef}. So there is a functor $\Phi: \mathrm{Fr}_*(k) \rightarrow \Cor$ which takes every $\sigma_V$ to the identity on $V$. 
Let $\cc M^{fr}$ be the category of pointed simplicial Nisnevich sheaves on $\mathrm{Fr}_*(k)$:
$\cc M^{fr} :=  \Delta^{op}\Shv(\mathrm{Fr}_*(k),Set_*)$.

$\Phi$ induces a forgetful functor $U_{\Phi} : \Delta^{op}\ShvA \rightarrow  \cc M^{fr}.$
The category $\cc M^{fr}$ is enriched in $\cc M$ where for $X, Y \in \cc M^{fr}$ the enriched morphism object $\cc M^{fr}(X,Y) \in \cc M$ is defined on $Z \in \Smk$ and $[n] \in \Delta^{op}$ by
$$\cc M^{fr}(X,Y)(Z)_n := \mathrm{Hom}_{M^{fr}}(X,Y(Z \times \Delta^n \times -)).$$
We have a monoidal adjunction $L_{\cc M} : \cc M \rightleftarrows \Delta^{op}\ShvA : U_{\cc M} $,
where the right adjoint $U_{\cc M}$ is the forgetful functor.
For $X, Y \in \Delta^{op}\ShvA$ we have a canonical map
$$U_{\cc M}(\inthom{\Delta^{op}\ShvA}{X}{Y}) \rightarrow \cc M^{fr}(U_{\Phi}(X), U_{\Phi(Y)}) $$
defined on $Z \in \Smk$ and $[n] \in \Delta^{op}$ by the map
\begin{multline*}
	U_{\cc M}(\inthom{\Delta^{op}\ShvA}{X}{Y})(Z)_n = \mathrm{Hom}_{\Delta^{op}\ShvA}(X,Y(Z \times \Delta^n \times -)) \overset{U_{\Phi}}{\rightarrow}\\ \rightarrow \mathrm{Hom}_{\cc M^{fr}}(U_{\Phi}(X),U_{\Phi}(Y)(Z \times \Delta^n \times -)) .
\end{multline*}
Let $\mathrm{Sm}/k_{+}$ be the category of framed correspondences of level 0 as defined in \cite[Example 2.4]{garkusha2019framed}.
Its morphism objects are defined by
$$\mathrm{Sm}/k_+(X,Y) := \inthom{\cc M}{X_+}{Y_+}.$$
Since $L_{\cc M}$ is lax monoidal, we have for every $X, Y \in \Smk$ a canonical map 
$L_{\cc M}(\mathrm{Sm}/k_{+}(X,Y)) \rightarrow \Sm(X,Y)$ in $\Delta^{op}\ShvA$,
which induces by adjunction a canonical map
$\mathrm{Sm}/k_{+}(X,Y) \rightarrow U_{\cc M}(\Sm(X,Y))$ in $\cc M$.
For every enriched motivic $\Cor$-space $\mathcal{X}$ we can now define a $\cc M$-enriched functor
$$\mathrm{Sm}/k_+ \rightarrow \cc M^{fr},\quad V \mapsto U_{\Phi}(\mathcal{X}(V)).$$
It acts on morphism sets via the composite
\begin{multline*}
	\mathrm{Sm}/k_+ (X,Y) \rightarrow U_{\cc M}(\Sm(X,Y)) \rightarrow U_{\cc M}(\inthom{\Delta^{op}(\ShvA)}{\mathcal{X}(X)}{\mathcal{X}(Y)}) \rightarrow\\ \rightarrow \cc M^{fr}(U_{\Phi}(\mathcal{X}(X)), U_{\Phi}(\mathcal{X}(Y))) .
\end{multline*}
With this enriched functor we can then also define a framed motivic $\Gamma$-space $\EM^{fr}(\mathcal{X})$ in the sense of \cite[Definition 3.5]{garkusha2019framed} by defining
$$\EM^{fr}(\mathcal{X}) : \Gamma^{op} \times \mathrm{Sm}/k_{+} \rightarrow \cc M^{fr},
\quad\EM^{fr}(\mathcal{X})(n_+,U) = U_{\Phi}(\mathcal{X}(U))^n.$$

\begin{prop} \label{framedgamma}
 Suppose that $\Cor$ has framed correspondences in the sense of Definition \ref{framedCorDef}.
	For every special enriched motivic $\Cor$-space $\mathcal{X}$ the framed motivic $\Gamma$-space $$\EM^{fr}(\mathcal{X}) : \Gamma^{op} \times \mathrm{Sm}/k_{+} \rightarrow \cc M^{fr},
	\quad\EM^{fr}(\mathcal{X})(n_+,U) = U_{\Phi}(\mathcal{X}(U))^n,$$ is a very special framed motivic $\Gamma$-space in the sense of \cite[Axioms 1.1]{garkusha2019framed}.
\end{prop}
\begin{proof}
	We verify the axioms 1)-5) and 7) for very special motivic $\Gamma$-spaces from \cite[Axioms 1.1]{garkusha2019framed}.
	For Axiom 1) we need to check that $\EM^{fr}(\mathcal{X})(0_+,U)=0$, $\EM^{fr}(\mathcal{X})(n_+,\emptyset)=0$
	and that $$\EM^{fr}(\mathcal{X})(n_+,U) \rightarrow \underset{i=1}{\overset{n}{\prod}} \EM^{fr}(\mathcal{X})(1_+,U)$$
	is a local equivalence.
	We have that $\EM^{fr}(\mathcal{X})(0_+,U)= U_{\Phi}(\mathcal{X}(U))^0 = 0$, and $$\EM^{fr}(\mathcal{X})(n_+,U) = U_{\Phi}(\mathcal{X}(U))^n \rightarrow \underset{i=1}{\overset{n}{\prod}} \EM^{fr}(\mathcal{X})(1_+,U)$$ is an isomorphism.
	According to \cite[Lemma 6.2]{bonart2022paper1}
	we have that $\mathcal{X}(\emptyset) = 0$.
	This implies that $\EM^{fr}(\mathcal{X})(n_+,\emptyset)=0$, hence Axiom 1) holds.
	
	Axioms 2)-5) for motivic $\Gamma$-spaces follow directly from axioms 1)-4) of special enriched motivic $\Cor$-spaces,
	except that for Axiom 2) we need to explain why the presheaf of stable homotopy groups
	$$V \mapsto \pi_n^s\EM^{fr}(\mathcal{X})(\bb S,U)(V) $$
	is radditive and $\sigma$-stable. 
	The $\sigma$-stability follows from the fact that $\Phi: \mathrm{Fr}_*(k) \rightarrow \Cor$ sends $\sigma_V$ to the identity.
	Let us now check that it is radditive.
	For every $U \in \Smk$, we have that $\mathcal{X}(U)$ is a sheaf of simplicial abelian groups. This implies that $\EM^{fr}(\mathcal{X})(\bb S,U)$ is a sheaf of $S^1$-spectra. So we have isomorphisms of $S^1$-spectra
	$\EM^{fr}(\mathcal{X})(\bb S,U)(\emptyset) =0$
	and 
	$$\EM^{fr}(\mathcal{X})(\bb S,U)(V_1 \coprod V_2) \cong \EM^{fr}(\mathcal{X})(\bb S,U)(V_1) \times \EM^{fr}(\mathcal{X})(\bb S,U)(V_2).$$
	Since stable homotopy groups $\pi_n^s$ preserve products and zero objects, it follows that 
	$$V \mapsto \pi_n^s\EM^{fr}(\mathcal{X})(\bb S,U)(V) $$ is radditive.
	Axiom 7) follows from the fact that $\mathcal{X}$ lands in sheaves of abelian groups.
\end{proof}

\begin{lem} \label{gammaeff}
Suppose that $\Cor$ has framed correspondences in the sense of Definition \ref{framedCorDef}.
	Let $\mathcal{X} $ be an enriched motivic $\Cor$-space and let $\EM^{fr}(\mathcal{X}) $ be its associated framed motivic $\Gamma$-space from Proposition \ref{framedgamma}.
	Then $\mathcal{X}$ is very effective in the sense of Definition \ref{suslindef} if and only if $\EM^{fr}(\mathcal{X})$ is very effective in the sense of \cite[Axioms 1.1]{garkusha2019framed}.
\end{lem}
\begin{proof}
	This follows directly from the definitions of effectiveness for $\mathcal{X}$ and $\EM^{fr}(\mathcal{X})$.
\end{proof}

\section{Enriched functors of chain complexes}

In this paper we freely use the canonical isomorphism of categories $\Ch([\Sm,\ShvA])\cong[\Sm,\Ch(\ShvA)]$ constructed in~\cite{garkusha2019derived}.
Likewise, there is a canonical isomorphism of categories $\Delta^{op}([\Sm,\ShvA])\cong[\Sm,\Delta^{op}(\ShvA)]$. In what follows we shall
freely use this isomorphism.

In the previous section we associated framed motivic $\Gamma$-spaces to enriched motivic $\Cor$-spaces. 
In this section we associate $\Ch(\ShvA)$-enriched functors in $[\Sm,\Ch(\ShvA)]$ to enriched motivic $\Cor$-spaces.

\begin{defs}\label{LambdaDef}
	Let $\mathcal{X} $ be an special enriched motivic $\Cor$-space and let
	$$DK^{-1}:\Delta^{op}[\Sm,\ShvA] \to \mathrm{Ch_{\geq0}}([\Sm,\ShvA])$$ 
	be the normalized Moore complex functor from the Dold-Kan correspondence.
	Denote by $\Lambda$ the composite functor
	$$ \Delta^{op}[\Sm,\ShvA]\xrightarrow{DK^{-1}}\mathrm{Ch_{\geq0}}([\Sm,\ShvA]) \rightarrow \Ch([\Sm,\ShvA]) .$$
\end{defs}

\begin{prop}\label{4properties}
	Let $\mathcal{X}\in \Delta^{op}[\Sm,\ShvA]$ be an enriched motivic $\Cor$-space.
	Then $\mathcal{X}$ is special if and only if $\Lambda(\mathcal{X})$ is in $DM_{\Cor}[\Sm]$, where the latter category is
	defined in \cite[Section~4]{bonart2022paper1}.
\end{prop}

\begin{proof}
	Four axioms defining special enriched motivic $\Cor$-spaces correspond to four properties of functors in $DM_{\Cor}[\Sm]$. 
	More precisely, the following four properties are true.
		
	(1)	$\mathcal{X}$ satisfies axiom $(1)$ of special enriched motivic $\Cor$-spaces if and only if for every $U \in \Smk$ the complex of sheaves $\Lambda(\mathcal{X})(U)$ has $\bb A^1$-invariant cohomology sheaves.
		
		(2) $\mathcal{X}$ satisfies the cancellation axiom $(2)$ if and only if $\Lambda(\mathcal{X})$ satisfies cancellation in the sense of \cite[Definition 4.6]{bonart2022paper1}.
		
		(3) $\mathcal{X}$ satisfies the $\bb A^1$-invariance axiom $(3)$ if and only if $\Lambda(\mathcal{X})$ is covariantly $\bb A^1$-invariant in the sense that $\Lambda(\mathcal{X})(U \times \bb A^1) \rightarrow \Lambda(\mathcal{X})(U)$ is a local quasi-isomorphism.
		
		(4) $\mathcal{X}$ satisfies the Nisnevich excision axiom $(4)$ if and only if $\Lambda(\mathcal{X})$ satisfies Nisnevich excision in the sense of \cite[Definition 4.9]{bonart2022paper1}. Here the functor $DK^{-1}: \Delta^{op}\ShvA \rightarrow \Ch_{\geq 0}(\ShvA)$ preserves homotopy cartesian squares for the following reason: Since $DK^{-1}$ preserves all weak equivalences, it is naturally weakly equivalent to its right derived functor $\mathbf{R}DK^{-1}$, and by \cite[Proposition 4.10]{benjamin2017homotopy} the right derived functor $\mathbf{R}DK^{-1}$ preserves all homotopy limits, including homotopy pullback squares.	
\end{proof}

\section{The Röndigs--\O stv\ae r Theorem}

Throughout this section $\mathcal{X}$ is a pointwise locally fibrant special enriched motivic $\Cor$-space.

\begin{dfn}\label{associatedbispectrum}
	We can extend $\mathcal{X}$ to an enriched functor
	$$\EM(\mathcal{X}): \Gamma^{op} \times \Sm \rightarrow \Delta^{op}\ShvA \quad (n_+,U) \mapsto \mathcal{X}(U)^n.$$
	We can take the $(S^1,\Gm)$-evaluation of $\EM(\mathcal{X})$ to get a motivic bispectrum $ev_{S^1,\Gm}(\EM(\mathcal{X})) \in SH(k)$.
	We define the \textit{bispectrum associated to $\mathcal{X}$ } to be this bispectrum
	$ev_{S^1,\Gm}(\mathcal{X}) := ev_{S^1,\Gm}(\EM(\mathcal{X})).$
	If $\Cor$ has framed correspondences,
	then $ev_{S^1,\Gm}(\mathcal{X})$ is also the evaluation of the framed motivic $\Gamma$-space $\EM^{fr}(\mathcal{X})$ from Proposition \ref{framedgamma}. Then by \cite[Section 2.7]{garkusha2019framed} the bispectrum $ev_{S^1,\Gm}(\mathcal{X}) = ev_{S^1,\Gm}(\EM^{fr}(\mathcal{X}))$ is a framed bispectrum in the sense of \cite[Definition 2.1]{garkusha2018triangulated}. In this case we say that $ev_{S^1,\Gm}(\mathcal{X})$ is the \textit{framed bispectrum associated to $\mathcal{X}$ }.
\end{dfn}

In this section we prove the following theorem extending Röndigs--\O stv\ae r's Theorem~\cite{rondigs2008modules}.

\begin{thm}\label{ROLemma}
	
	For every $U \in \Smk$ we have a natural isomorphism
	$$ev_{S^1,\Gm}(\mathcal{X}) \wedge \Sigma^{\infty}_{S^1,\Gm}U_+  \overset{\sim}{\rightarrow} ev_{S^1,\Gm}(\mathcal{X}(U \times -))  $$
	in $\SH(k)[1/p]$, where $p$ is the exponential characteristic of $k$.
\end{thm}

To prove it we will need a few lemmas.

For a finite pointed set $n_+=\{0,\dots,n\}$ and $U \in \Smk$ let $n_+ \otimes U$ be the $n$-fold coproduct $\underset{i=1}{\overset{n}{\coprod}} U$.
Let $f\cc M$ be the category of finitely presented motivic spaces in the sense of \cite{dundas2003motivic}.
Given an enriched motivic $\Cor$-space $\mathcal{X}$ we can define an extended functor $\hat{\mathcal{X}} : f\mathscr{M} \rightarrow \Delta^{op}\ShvA$ by

$$\hat{\mathcal{X}}(A)_n := \colim{(\Delta[m] \times U)_+ \rightarrow A^c} \mathcal{X}( \Delta[m]_{n,+} \otimes U)_n  $$
where $A^c$ is a cofibrant replacement of $A$ in $f\mathscr{M}$.
We have for all $U \in \Sm$ that $\hat{\mathcal{X}}(U) \cong \mathcal{X}(U)$ in $\Delta^{op}\ShvA$.

Let $ev_{S^1,\Gm}(\hat{\mathcal{X}})$ be the $(S^1,\Gm)$-evaluation bispectrum of the extended functor $\hat{\mathcal{X}} : f\mathscr{M} \rightarrow \Delta^{op}\ShvA$.

\begin{lem} \label{hatcomparison}
	We have a canonical isomorphism of motivic $(S^1,\Gm)$-bispectra
	$ev_{S^1,\Gm}(\hat{\mathcal{X}}) \cong ev_{S^1,\Gm}(\mathcal{X}) $
	between the $(S^1,\Gm)$-evaluation of the extended functor $\hat{\mathcal{X}}$, and the bispectrum associated with $\mathcal{X}$ in the sense of Definition \ref{associatedbispectrum}.
\end{lem}
\begin{proof}
	By \cite[Lemma 6.2]{bonart2022paper1} we have for all $U, V \in \Smk$ an isomorphism $\mathcal{X}(U \coprod V) \cong \mathcal{X}(U)\oplus \mathcal{X}(V)$
	in $\Delta^{op}\ShvA$.
	This implies that we have for all $U \in \Smk, n \geq 0$ an isomorphism
	$\mathcal{X}(n_+ \otimes U) \cong \underset{i=1}{\overset{n}{\bigoplus}} \mathcal{X}(U) = \EM(\mathcal{X})(n_+,U)$
	in $\Delta^{op}\ShvA$.
	We then compute for $A \in f\cc M$ that $$\hat{\mathcal{X}}(A)_n =  \colim{(\Delta[k] \times U)_+ \rightarrow A^c} \mathcal{X}( \Delta[k]_{n,+} \otimes U)_n  \cong \colim{(\Delta[k] \times U)_+ \rightarrow A^c} \EM(\mathcal{X})(\Delta[k]_{n,+},U)_n . $$
	So $\hat{\mathcal{X}}$ naturally extends $\EM(\mathcal{X})$ from $\Gamma^{op} \times \mathrm{Sm}/k_+$ to $f\mathscr{M}$.
	This then implies that
	$$ev_{S^1,\Gm}(\hat{\mathcal{X}}) \cong ev_{S^1,\Gm}(\EM(\mathcal{X})) = ev_{S^1,\Gm}(\mathcal{X}) $$
	as required.
\end{proof}

\begin{proof}[Proof of Theorem \ref{ROLemma}]\label{ROproof}
	Using Definition \ref{LambdaDef} we can associate to $\mathcal{X}$ an enriched functor $\Lambda(\mathcal{X}) : \Sm \rightarrow \Ch(\ShvA).$
	By Proposition \ref{4properties} the functor $\Lambda(\mathcal{X})$ is in $DM_{\Cor}[\Sm]$. By \cite[Proposition 4.13]{bonart2022paper1} this implies that $\Lambda(\mathcal{X})$ is strictly $\sim$-local in the sense of \cite[Definition 4.3]{bonart2022paper1}. Since $\mathcal{X}$ is pointwise locally fibrant, it follows that $\Lambda(\mathcal{X})$ is $\sim$-fibrant in the sense of \cite[Definition 4.11]{bonart2022paper1}.
	
	Using \cite[Section 7]{bonart2022paper1} we can associate to $\Lambda(\mathcal{X})$ an $\cc M$-enriched functor $\Lambda(\mathcal{X})^{\cc M} : f\cc M \rightarrow \Sp_{S^1}(\cc M).$
	We can take the $0$-th level of this functor to get a motivic functor $\Lambda(\mathcal{X})^{\cc M}_0 : f\cc M \rightarrow \cc M.$
	By \cite[Lemma 7.7]{bonart2022paper1} the motivic functor $\Lambda(\mathcal{X})^{\cc M}_0$ preserves motivic equivalences between cofibrant objects.
	By \cite[Appendix B, Corollary B.2]{levine2019algebraic} the suspension bispectrum $\Sigma^{\infty}_{S^1,\Gm}U_+$ is strongly dualizable in $\SH(k)[1/p]$. 
	From \cite[Lemma 7.2]{bonart2022paper1} it follows that we have an isomorphism
	$$ev_{S^1,\Gm}(\Lambda(\mathcal{X})^{\cc M}_0) \wedge \Sigma^{\infty}_{S^1,\Gm}U_+\cong ev_{S^1,\Gm}(\Lambda(\mathcal{X})^{\cc M}_0(U \times -))$$
	in $\SH(k)[1/p]$.
	To prove the theorem, we now just need to show that there is a natural isomorphism $$ev_{S^1,\Gm}(\Lambda(\mathcal{X})^{\cc M}_0) \rightarrow ev_{S^1,\Gm}(\mathcal{X})$$ in $\SH(k)$. For this we need some intermediate steps.
	Firstly, by Lemma \ref{hatcomparison} we have an isomorphism
	$$ev_{S^1,\Gm}(\hat{\mathcal{X}}) \rightarrow ev_{S^1,\Gm}(\mathcal{X}).$$
	So we now just need to find an isomorphism $$ev_{S^1,\Gm}(\Lambda(\mathcal{X})^{\cc M}_0) \rightarrow ev_{S^1,\Gm}(\hat{\mathcal{X}})$$
	in $\SH(k)$.
	
	In what follows, we let $$DK^{-1} : \Delta^{op}\ShvA \rightarrow \Ch_{\geq 0}(\ShvA)$$ be the Dold-Kan equivalence, i.e. the normalized Moore complex functor, for the Grothendieck category $\ShvA$.
	We let $$DK^{-1}_{\Ch(\ShvA)} : \Delta^{op}\Ch(\ShvA) \rightarrow \Ch_{\geq 0}(\Ch(\ShvA))$$ be the Dold-Kan correspondence for the Grothendieck category $\Ch(\ShvA)$.
	And we let $$DK^{-1}_{\mathrm{double}} : \Delta^{op}\Delta^{op}\ShvA \rightarrow \Ch_{\geq 0}(\Ch_{\geq 0}(\ShvA))$$ be the Dold-Kan correspondence applied twice, so that it takes bisimplicial objects to double complexes.
	
	Using \cite[Section 6, Equation (1)]{bonart2022paper1} we can extend $\Lambda(\mathcal{X})$ to a functor
	$$\widehat{\Lambda(\mathcal{X})} : f \cc M \rightarrow \Ch(\ShvA),$$
	$$\widehat{\Lambda(\mathcal{X})}(A) := \mathrm{Tot}(DK^{-1}_{\Ch(\ShvA)} ( \colim{(\Delta[k] \times U)_+ \rightarrow A^c} \Lambda(\mathcal{X})^{\Delta^{op}}( \Delta[k]_+ \otimes U)     )) .$$
	Now for every $A\in f\cc M$ we have a natural quasi-isomorphism
	$$DK^{-1}(\hat{\mathcal{X}}(A)) \rightarrow \widehat{\Lambda(\mathcal{X})}(A)$$
	in $\Ch(\ShvA)$ for the following reason:
	$\hat{\mathcal{X}}(A)$ is the diagonal of the bisimplicial sheaf $$\colim{(\Delta[m] \times U)_+ \rightarrow A^c} \mathcal{X}( \Delta[m]_{+} \otimes U).$$
	By \cite[page 37, equation 24]{cegarra2005diagonal}, or \cite[Theorem 2.9]{dold1961diagonal}, for every bisimplicial object $S \in \Delta^{op}\Delta^{op}\ShvA$ there is a quasi-isomorphism $$DK^{-1}(\mathrm{diag}(S)) \rightarrow  \mathrm{Tot}(DK^{-1}_{\mathrm{double}}(S))$$
	in $\Ch_{\geq 0}(\ShvA)$.
	So for every $A \in f\cc M$ there is a quasi-isomorphism
	$$ DK^{-1}(\hat{\mathcal{X}}(A)) \rightarrow \mathrm{Tot}(DK^{-1}_{\mathrm{double}}(\colim{(\Delta[m] \times U)_+ \rightarrow A^c} \mathcal{X}( \Delta[m]_{+} \otimes U))) \cong$$
	$$\cong \mathrm{Tot}(DK^{-1}_{\Ch(\ShvA)}(\colim{(\Delta[m] \times U)_+ \rightarrow A^c} DK^{-1}(\mathcal{X}( \Delta[m]_{+} \otimes U)))) \cong \widehat{\Lambda(\mathcal{X})}(A).$$
	By construction $\widehat{\Lambda(\mathcal{X})}$ lands in $\Ch_{\geq 0}(\ShvA)$, so we can take the functor $$DK \circ \widehat{\Lambda(\mathcal{X})} : f \cc M \rightarrow \Delta^{op}\ShvA$$ and form the naive  $(S^1,\Gm)$-evaluation bispectrum $$ev_{S^1,\Gm}(DK \circ \widehat{\Lambda(\mathcal{X})}) \in \SH(k).$$
	The above quasi-isomorphism, then induces an isomorphism
	$$ev_{S^1,\Gm}(DK \circ \widehat{\Lambda(\mathcal{X})}) \rightarrow ev_{S^1,\Gm}(\hat{\mathcal{X}}) $$
	in $\SH(k)$.
	So to prove the theorem we now just need an isomorphism
	$$ev_{S^1,\Gm}(\Lambda(\mathcal{X})^{\cc M}_0) \rightarrow ev_{S^1,\Gm}(DK \circ \widehat{\Lambda(\mathcal{X})})$$
	in $\SH(k)$.
	
	By \cite[Lemma 7.5]{bonart2022paper1}, for every $A \in f\cc M$ with cofibrant replacement $A^c$ we have an isomorphism
	$$\hat{U} \circ \widehat{\Lambda(\mathcal{X})}(A) \rightarrow \Lambda(\mathcal{X})^{\cc M}(A^c)$$
	in $\Sp_{S^1}(\cc M)$, where $\hat{U} : \Ch(\ShvA) \rightarrow \Sp_{S^1}(\cc M)$ is the canonical functor defined in \cite[Section 7]{bonart2022paper1}.
	Let $ev_0 : \Sp_{S^1}(\cc M) \rightarrow \cc M$ be the functor taking the $0$-th level of a $S^1$-spectrum. So $\Lambda(\mathcal{X})^{\cc M}_0 = ev_0 \circ \Lambda(\mathcal{X})^{\cc M}.$ By the proof of \cite[Lemma 7.4]{bonart2022paper1}, the functor $ev_0\circ \hat{U}$ is isomorphic to the composite
	$$\Ch(\ShvA) \overset{\tau_{\geq 0}}{\rightarrow} \Ch_{\geq 0}(\ShvA) \overset{DK}{\rightarrow} \Delta^{op}\ShvA \overset{U}{\rightarrow} \cc M,$$
	where $\tau_{\geq 0}$ is the good truncation functor and $U$ is the forgetful functor.
	Since $\widehat{\Lambda(\mathcal{X})}$ lands in $\Ch_{\geq 0}(\ShvA)$, it does not get changed by truncation. So we get that
	$$ev_0 \circ \hat{U} \circ \widehat{\Lambda(\mathcal{X})} \cong U \circ DK \circ \widehat{\Lambda(\mathcal{X})}.$$
	So for every $A \in f\cc M$ we have a natural isomorphism
	$$( U \circ DK \circ \widehat{\Lambda(\mathcal{X})})(A) \rightarrow \Lambda(\mathcal{X})^{\cc M}_0(A^c)$$
	in $\cc M$. Since $S^1$ and $\Gm$ are cofibrant in $f\cc M$, we get an isomorphism
	$$ev_{S^1,\Gm}(\Lambda(\mathcal{X})^{\cc M}_0) \rightarrow ev_{S^1,\Gm}(DK \circ \widehat{\Lambda(\mathcal{X})})$$
	in $\SH(k)$, as claimed.
	
	Putting it all together, we get a commutative diagram
	$$\xymatrix{ev_{S^1,\Gm}(\Lambda(\mathcal{X})^{\cc M}_0) \wedge \Sigma^{\infty}_{S^1,\Gm}U_+ \ar[r]^\sim \ar[d]^\sim &  ev_{S^1,\Gm}(\Lambda(\mathcal{X})^{\cc M}_0(U \times -)) \ar[d]^\sim \\ 
	ev_{S^1,\Gm}(DK \circ \widehat{\Lambda(\mathcal{X})}) \wedge \Sigma^{\infty}_{S^1,\Gm}U_+ \ar[r] \ar[d]^\sim & ev_{S^1,\Gm}(DK \circ \widehat{\Lambda(\mathcal{X})}(U \times -)) \ar[d]^\sim \\
	ev_{S^1,\Gm}(\hat{\mathcal{X}}) \wedge \Sigma^{\infty}_{S^1,\Gm}U_+ \ar[r] \ar[d]^\sim & ev_{S^1,\Gm}(\hat{\mathcal{X}}(U \times -)) \ar[d]^\sim \\ ev_{S^1,\Gm}(\mathcal{X}) \wedge \Sigma^{\infty}_{S^1,\Gm}U_+ \ar[r] & ev_{S^1,\Gm}(\mathcal{X}(U \times -)) }$$
in which all the vertical maps and the top horizontal map are isomorphisms in $\SH(k)[1/p]$. It follows that the bottom horizontal map is also an isomorphism in $\SH(k)[1/p]$. This completes the proof.
\end{proof}

\section{A motivic model structure for enriched motivic $\Cor$-spaces}\label{motmodel}

In Section \ref{sectionmodelstructures} we showed that $\Delta^{op}\ShvA$ with the degreewise tensor product $\otimes$ has a model structure that is cellular, weakly finitely generated, monoidal, strongly left proper and satisfies the monoid axiom (see Proposition \ref{propmodelcategory}).
We can apply \cite[Theorem 4.2]{dundas2003enriched} to this model structure to get a weakly finitely generated model structure on the category of enriched functors $[\Sm,\Delta^{op}\ShvA]$ in which the weak equivalences, respectively fibrations, are the $\Sm$-pointwise local equivalences, respectively $\Sm$-pointwise local fibrations. We call this the \textit{local model structure} on $[\Sm,\Delta^{op}\ShvA]$.
By \cite[Theorem 4.4]{dundas2003enriched} the local model structure on $[\Sm,\Delta^{op}\ShvA]$ is monoidal with the usual Day convolution product.
By \cite[Corollary 4.8]{dundas2003enriched} the local model structure on $[\Sm,\Delta^{op}\ShvA]$ is left proper.
Since $[\Sm,\Delta^{op}\ShvA]$ is weakly finitely generated, and all cofibrations in $\Delta^{op}\ShvA$ are monomorphisms, it follows that $[\Sm,\Delta^{op}\ShvA]$ is cellular.
Note that for every $U \in \Smk$ the representable functor $\Sm(U,-) \cong \Sm(U,-) \otimes pt$ is cofibrant in $[\Sm,\Delta^{op}\ShvA]$.

In this section we define another model structure on $[\Sm,\Delta^{op}\ShvA]$ such that the fibrant objects are the pointwise locally fibrant special enriched motivic $\Cor$-spaces.

\begin{defs}
	Similarly to \cite[Section 4]{bonart2022paper1}
	we define four families of morphisms in the category $[\Sm,\Delta^{op}\ShvA]$.
	\begin{enumerate}
		\item We let $\bb A^1_1$ be the family of morphisms consisting of
		$$\Sm(U,-) \otimes \bb A^1 \rightarrow \Sm(U,-)$$
		for every $U \in \Smk$.
		\item We let $\tau$ be the family of morphisms consisting of the evaluation map
		$$\Sm(\Gmn{n+1}\times U,-) \otimes \Gmn{1} \rightarrow \Sm(\Gmn{n}\times U,-) $$
		for every $n \geq 0$ and $U \in \Smk$.
		\item
		We let $\bb A^1_2$ be the family of morphisms consisting of
		$$\Sm(U,-) \rightarrow \Sm(U \times \bb A^1,-) $$
		for every $U \in \Smk$.
		\item
		We let $Nis$ be the following family of morphisms: For every elementary Nisnevich square $Q$ of the form
		$$\xymatrix{ U^\prime \ar[r]^\beta \ar[d]_\alpha & X^\prime \ar[d]^\gamma\\
			X \ar[r]_\delta & X }$$
		in $\Smk$ we have a square
		$$\xymatrix{ \Sm(U^\prime,-)  &\ar[l]_{\beta^*} \Sm(X^\prime,-) \\
			\Sm(U,-) \ar[u]^{\alpha^*} & \ar[l]^{\delta^*} \ar[u]_{\gamma^*} \Sm(X,-) } $$
		in $[\Sm,\Delta^{op}\ShvA]$, which induces a map on homotopy fibers
		$$p_Q: \mathrm{hofib}(\gamma^*) \rightarrow \mathrm{hofib}(\alpha^*). $$
		We let $Nis$ be the family of morphisms consisting of $p_Q$ for every elementary Nisnevich square $Q$.
	\end{enumerate}
	
	Finally, we let $\sim$ denote the union of all these four classes of morphisms.
	$$\sim := \bb A^1_1 + \tau + \bb A^1_2 + Nis.$$
\end{defs}

\begin{defs}\label{defstrictlylocal}
	For $X, Y \in [\Sm,\Delta^{op}\ShvA]$ let
	$$\mathrm{map}^{\Delta^{op}\ShvA}(X,Y) \in \Delta^{op}\ShvA$$
	be the simplicial sheaf of morphisms from $X$ to $Y$. It is defined by taking the internal hom $\inthom{[\Sm,\Delta^{op}\ShvA]}{X}{Y}$ and evaluating it at the point $pt \in \Sm$.
	$$\mathrm{map}^{\Delta^{op}\ShvA}(X,Y) := \inthom{[\Sm,\Delta^{op}\ShvA]}{X}{Y}(pt).$$
	For $U \in \Smk$ and $n \geq 0$ we have
	$$\mathrm{map}^{\Delta^{op}\ShvA}(X,Y)(U)_n = \mathrm{Hom}_{[\Sm,\Delta^{op}\ShvA]}(X \otimes U \otimes \Delta[n], Y)$$
	in $\Ab$.
	
	Similarly to \cite[Definition 4.3]{bonart2022paper1}, given a class of morphisms $S$ in $[\Sm,\Delta^{op}\ShvA]$ and an object $X \in [\Sm,\Delta^{op}\ShvA]$ with pointwise locally fibrant replacement $X^f$ we say that $X$ is \textit{strictly $S$-local} if for every $s : A \rightarrow B$ with $s \in S$ the morphism
	$$s^* : \mathrm{map}^{\Delta^{op}\ShvA}(B,X^f) \rightarrow \mathrm{map}^{\Delta^{op}\ShvA}(B,X^f)$$
	is a local quasi-isomorphism of sheaves.
\end{defs}

\begin{lem}
	A enriched motivic $\Cor$-space $\mathcal{X} : \Sm \rightarrow \ShvA$ is special if and only if it is strictly $\sim$-local.
\end{lem}
\begin{proof}
	By Lemma \ref{4properties} we know that $\mathcal{X}$ is special if and only if $\Lambda(\mathcal{X})$ lies in $DM_{\Cor}[\Sm]$.
	By \cite[Proposition 4.13]{bonart2022paper1} this is the case if and only if $\Lambda(\mathcal{X})$ is strictly $\sim$-local in the sense of \cite[Definition 4.3]{bonart2022paper1},
	and this is the case if and only if $\mathcal{X}$ is strictly $\sim$-local in the sense of Definition \ref{defstrictlylocal}.
\end{proof}

\begin{defs}
	Given a class of morphisms $S$ in $[\Sm,\Delta^{op}\ShvA]$, we write $\widehat{S}$ for the class of morphisms
	$$\widehat{S} := \{s \otimes Z \mid s \in S, Z \in \Smk \}.$$
	
	We define the \textit{enriched motivic model structure} on $[\Sm,\Delta^{op}\ShvA]$ to be the left Bousfield localization of the local model structure on $[\Sm,\Delta^{op}\ShvA]$ with respect to the class of morphisms $\widehat{\sim}$. This model category will be denoted by $[\Sm,\Delta^{op}\ShvA]_{\mot}$.
\end{defs}

\begin{lem}
	Let $S$ be a class of morphisms in $[\Sm,\Delta^{op}\ShvA]$ with cofibrant domains and codomains. Then an object $F \in [\Sm,\Delta^{op}\ShvA]$ is strictly $S$-local if and only if its local fibrant replacement $F^f$ is $\widehat{S}$-local in the usual model category theoretic sense of \cite[Definition 3.1.4]{hirschhorn2003model}.
\end{lem}
\begin{proof}
	Let $F^f$ be a pointwise locally fibrant replacement of $F$.
	For every $s: A \rightarrow B, s \in \widehat{S}$ let $s^c : A^c \rightarrow B^c$ be a cofibrant replacement of $s$. This means we have a commutative square
	$$\xymatrix{A^c \ar[d] \ar[r]^{s^c} & B^c  \ar[d] \\
		A \ar[r]^s & B}$$
	such that the vertical maps are trivial fibrations, $A^c$ and $B^c$ are cofibrant and $s^c$ is a cofibration.
	
	Note that for every $s \in \widehat{S}$ the domain $A$ and codomain $B$ are already cofibrant, but $s$ is not neccessarily a cofibration.
	
	For $X, Y \in [\Sm,\Delta^{op}\ShvA]$ let $\mathrm{map}^{\Delta^{op}\Set}(X,Y) \in \Delta^{op}\Set$ denote the non-derived simplicial mapping space. It can be defined by
	$$\mathrm{map}^{\Delta^{op}\Sets}(X,Y) := \inthom{[\Sm,\Delta^{op}\ShvA]}{X}{Y}(pt)(pt).$$
	
	Now $F^f$ is $\widehat{S}$-local in the usual model category theoretic sense if and only if for every $s \in \widehat{S}$ the map
	$$s^{c,*} : \mathrm{map}^{\Delta^{op}\Sets}(B^c,F^f) \rightarrow \mathrm{map}^{\Delta^{op}\Sets}(A^c,F^f)$$
	is a weak equivalence.
	
	We have a commutative square
	$$\xymatrix{  \mathrm{map}^{\Delta^{op}\Set}(B,F^f) \ar[r]^{s^*} \ar[d] &   \mathrm{map}^{\Delta^{op}\Set}(A,F^f) \ar[d] \\ \mathrm{map}^{\Delta^{op}\Set}(B^c,F^f) \ar[r]^{s^{c,*}} &   \mathrm{map}^{\Delta^{op}\Set}(A^c,F^f)}$$
	
	Since the functor $\mathrm{map}^{\Delta^{op}\Sets}(-,F^f)$ sends trivial cofibrations to trivial fibrations, it follows by Ken Brown's lemma \cite[Lemma 1.1.12]{hovey2007model}, that  $\mathrm{map}^{\Delta^{op}\Sets}(-,F^f)$ sends weak equivalences between cofibrant objects to weak equivalences.
	Since the maps $A^c \rightarrow A$ and $B^c \rightarrow B$ are weak equivalences between cofibrant objects, it follows that the vertical maps in the above commutative diagram are weak equivalences.
	Therefore $F^f$ is $\widehat{S}$-local if and only if for every $s \in \widehat{S}$ the map
	$$s^* : \mathrm{map}^{\Delta^{op}\Sets}(B,F^f) \rightarrow \mathrm{map}^{\Delta^{op}\Sets}(A,F^f)$$
	is a weak equivalence.
	Every $s \in \widehat{S}$ is of the form $t \otimes Z$ for some $Z \in \Smk$ and $t : C \rightarrow D$ with $t \in S$.
	We have a commutative diagram in which the vertical maps are isomorphisms:
	$$\xymatrix{ \mathrm{map}^{\Delta^{op}\Sets}(D \otimes Z,F^f) \ar[r]^{(t \otimes Z)^*} \ar[d]_\sim  & \mathrm{map}^{\Delta^{op}\Sets}(C \otimes Z,F^f) \ar[d]^\sim \\
		\mathrm{map}^{\Delta^{op}\ShvA}(D,F^f)(Z) \ar[r]^{t^*} &  	\mathrm{map}^{\Delta^{op}\ShvA}(C,F^f)(Z)} $$
	So $F^f$ is $\widehat{S}$-local if and only if for every $t : C \rightarrow D, t \in S$ the map
	$$t^* : \mathrm{map}^{\Delta^{op}\ShvA}(D,F^f) \rightarrow \mathrm{map}^{\Delta^{op}\ShvA}(C,F^f)$$
	is a sectionwise weak equivalence in $\Delta^{op}\ShvA$.
	Since $C$, $D$ are cofibrant and $F^f$ is locally fibrant, the domain and codomain of $t^*$ are fibrant. So $t^*$ is a sectionwise weak equivalence if and only if it is a local weak equivalence.
	Therefore $F^f$ is $\widehat{S}$-local if and only if $F$ is strictly $S$-local.
\end{proof}

So the fibrant objects of $[\Sm,\Delta^{op}\ShvA]_{\mot}$ are the pointwise locally fibrant special enriched motivic $\Cor$-spaces.

\begin{dfn}
	Let $\cc D([\Sm,\Delta^{op}\ShvA])$ be the homotopy category of $[\Sm,\Delta^{op}\ShvA]$ with respect to the pointwise local model structure.
	Define $\mathrm{Spc}_{\Cor}[\Sm]$ as the full subcategory of $\cc D([\Sm,\Delta^{op}\ShvA])$ consisting of special enriched motivic $\Cor$-spaces.
\end{dfn}

We document above lemmas as follows.

\begin{thm}\label{corSpcAembedding}
	 The category $\mathrm{Spc}_{\Cor}[\Sm]$ is equivalent to the homotopy category of the model category $[\Sm,\Delta^{op}\ShvA]_{\mot}$.
The fibrant objects of $[\Sm,\Delta^{op}\ShvA]_{\mot}$ are the pointwise locally fibrant special enriched motivic $\Cor$-spaces.
\end{thm}

The preceding theorem is also reminiscent of Bousfield--Friedlander's theorem~\cite{BFgamma} stating that
fibrant objects in the model category of classical $\Gamma$-spaces are given by very special $\Gamma$-spaces.

\section{Reconstructing $DM_{\Cor, \geq 0}^{\eff}$}

\begin{defs}\label{Gmmotivedef}
	For $U \in \Smk$ define $\motive{U} \in DM_{\Cor}$ by
	$$\motive{U} := (M_{\Cor}(U \times \Gmn{n} ))_{n \geq 0} ,$$
	where $M_{\Cor}(X) := C_*\Cor(-,X)_{\nis}$ is the $\Cor$-motive of $X$.
	We call $\motive{U}$ the \textit{big $\Cor$-motive} of $U$.
\end{defs}

Let $\mathcal{U} : DM_{\Cor} \rightarrow \SH(k)$ be the forgetful functor, and let $\mathcal{L} : \SH(k) \rightarrow DM_{\Cor}$ be its left adjoint.
\begin{lem}\label{motivecomparison}
	The natural morphism $$\mathcal{L} (\Sigma^{\infty}_{S^1,\Gm} U_+) \rightarrow \motive{U} $$ is an isomorphism in $DM_{\Cor}$.
\end{lem}
\begin{proof}
	In weight $n$ this morphism is the motivic equivalence
	$$\Cor(-,U)_{\nis} \rightarrow C_*\Cor(-,U)_{\nis} = M_{\Cor}(U) .$$
	So the map $$\mathcal{L} (\Sigma^{\infty}_{S^1,\Gm} U_+) \rightarrow \motive{U}$$ is a levelwise motivic equivalence, and therefore an isomorphism in $DM_{\Cor}$.
\end{proof}

Let $DM_{\Cor, \geq 0}$ be the full subcategory of $DM_{\Cor}$ consisting of those $\Gm$-spectra of chain complexes which are connective chain complexes in each weight. 
Note that by construction, for every $U \in \Smk$ we have $\motive{U} \in DM_{\Cor, \geq 0}$.

\begin{thm} \label{geqtheorem}
	The naive $\Gm$-evaluation functor is an equivalence of categories
	$$ev_{\Gm} :\mathrm{Spc}_{\Cor}[\Sm] \rightarrow DM_{\Cor, \geq 0} .$$
\end{thm}
\begin{proof}
	Since the exponential characteristic $p$ of $k$ is invertible in $\Cor$, it follows from \cite[Theorem 4.14]{bonart2022paper1}
	that the naive $\Gm$-evaluation functor is an equivalence of categories
	$$ev_{\Gm}: DM_{\Cor}[\Sm] \rightarrow DM_{\Cor} .$$
	Here $DM_{\Cor}[\Sm]$ consists of those enriched functors $F : \Sm \rightarrow \Ch(\ShvA)$ which satisfy
	contravariant $\bb A^1$-invariance, cancellation,  covariant $\bb A^1$-invariance and Nisnevich excision  (see \cite[Section 4]{bonart2022paper1} for details).
	
	Let $DM_{\Cor}[\Sm]_{\geq 0}$ be the full subcategory of $DM_{\Cor}[\Sm]$ on those functors $F : \Sm \rightarrow \Ch(\ShvA)$ which factor over $\Ch_{\geq 0}(\ShvA)$.
	The equivalence $ev_{\Gm}$ restricts to a fully faithful functor on connective chain complexes
	$$ev_{\Gm, \geq 0}: DM_{\Cor}[\Sm]_{\geq 0} \rightarrow DM_{\Cor, \geq 0} .$$
	The functor $ev_{\Gm} : \mathrm{Spc}_{\Cor}[\Sm] \rightarrow DM_{\Cor, \geq 0}$ of the theorem will factor through $ev_{\Gm, \geq 0}$. 
	We claim that this restricted $\Gm$-evaluation functor $ev_{\Gm, \geq 0}$ is an equivalence. Since it is fully faithful we only need to show essential surjectivity.
	
	Take $F \in DM_{\Cor, \geq 0}$. Since $ev_{\Gm}$ is essentially surjective on non-connective chain complexes, there exists $G \in DM_{\Cor}[\Sm]$ such that $ev_{\Gm}(G) \cong F$.
	Let $$\tau_{\geq 0} : \Ch([\Sm,\ShvA]) \rightarrow \Ch_{\geq 0}([\Sm,\ShvA])$$ be the good truncation functor for chain complexes of the Grothendieck category of enriched functors $[\Sm,\ShvA]$. Also denote by $\tau_{\geq 0} : \SpGm(\Ch(\ShvA)) \rightarrow \SpGm(\Ch_{\geq 0}(\ShvA))$ the good truncation functor of $\Ch(\ShvA)$ applied in each weight.
	
	Consider the commutative diagram
	$$\xymatrix{
		ev_{\Gm}(\tau_{\geq 0}(G))  \ar[d]\ar[r] & \tau_{\geq 0}(F) \ar[d]^\sim \\ ev_{\Gm}(G) \ar[r]^\sim & F }$$
	We know that the bottom horizontal map and the right vertical map are isomorphisms in $DM_{\Cor}$.
	We claim that $\tau_{\geq 0}(G) \rightarrow G$ is an isomorphism in $D([\Sm,\ShvA])$. For this it suffices to show that for every $U \in \Smk$ the negative homology sheaves of $G(U)$ are zero.
	
	We have a chain of isomorphisms in $D(\ShvA)$ $$G(U) \cong G(U\times pt) = ev_{\Gm}(G(U \times -))(0)$$

	By  \cite[Theorem 7.1]{bonart2022paper1} we have  isomorphisms in $DM_{\Cor}$
	$$ev_{\Gm}(G(U \times -)) \cong ev_{\Gm}(G) \wedge \motive{U}  \cong F \wedge   \motive{U} .$$
	Since  $DM_{\Cor,\geq 0}$ is closed under the smash product of $DM_{\Cor}$, we have that  $F \wedge   \motive{U} \in DM_{\Cor,\geq 0}$.
	
	Therefore $G(U) = ev_{\Gm}(G(U \times -))(0)$ has vanishing negative homology sheaves. 
	So $\tau_{\geq 0}(G) \rightarrow G$ is an isomorphism in $D([\Sm,\ShvA])$, and then it follows that the composite map
	$$ev_{\Gm}(\tau_{\geq 0}(G)) \rightarrow ev_{\Gm}(G) \rightarrow F$$ is an isomorphism in $DM_{\Cor}$.
	So 
	$$ev_{\Gm, \geq 0}: DM_{\Cor}[\Sm]_{\geq 0} \rightarrow DM_{\Cor, \geq 0} $$ is essentially surjective, and hence an equivalence.
	
	Let $\cc D([\Sm,\Ch_{\geq 0}(\ShvA)])$ be the homotopy category of $[\Sm,\Ch_{\geq 0}(\ShvA)]$ with respect to the local model structure.
	The Dold-Kan correspondence induces an equivalence of categories
	$$\Lambda: \cc D([\Sm,\Delta^{op}\ShvA]) \rightarrow \cc D([\Sm,\Ch_{\geq 0}(\ShvA)]) .$$
	From Proposition  \ref{4properties} it now follows that
	we have a commutative diagram
	$$\xymatrix{ \cc D([\Sm,\Delta^{op}\ShvA]) \ar[r]^(.45){\Lambda} & \cc D([\Sm,\Ch_{\geq 0}(\ShvA)]) \\
		\mathrm{Spc}_{\Cor}[\Sm] \ar[u] \ar[r] & DM_{\Cor}[\Sm]_{\geq 0 } \ar[u] } $$
	where the vertical maps are the inclusion maps.
	Proposition  \ref{4properties} implies that the bottom horizontal arrow is essentially surjective.
	Since the the vertical maps and the top horizontal map are also fully faithful, we know that the bottom horizontal map is fully faithful, so it is an equivalence of categories.
	So we get an equivalence of categories $$ev_{\Gm} : \mathrm{Spc}_{\Cor}[\Sm] \rightarrow DM_{\Cor, \geq 0} $$ as was to be shown.
\end{proof}

From now on assume that $\Cor$ has framed correspondences in the sense of Definition \ref{framedCorDef}.

\begin{prop}\label{effcomparsion}
	Let $\mathcal{X}$ be a special enriched motivic $\Cor$-space.
	Let $ev_{S^1,\Gm}(\mathcal{X}) \in \SH(k)_{\nis}^{fr}$ be its associated framed bispectrum, as in Definition \ref{associatedbispectrum}.
	Then $ev_{S^1,\Gm}(\mathcal{X})$ is effective, in the sense of \cite[Definition 3.5]{garkusha2018triangulated} if and only if $\mathcal{X}$ is very effective, in the sense of Definition \ref{suslindef}.
\end{prop}
\begin{proof}
	Suppose that $\mathcal{X}$ is very effective.
	By Lemma \ref{gammaeff} the enriched motivic $\Cor$-space $\mathcal{X}$ is very effective if and only if the associated framed motivic $\Gamma$-space $\EM(\mathcal{X})$ is very effective.
	If $\EM(\mathcal{X})$ is very effective, then this clearly implies that the framed bispectrum $ev_{S^1,\Gm}(\mathcal{X})$, from Definition \ref{associatedbispectrum}, is very effective in the sense of \cite[Definition 3.5]{garkusha2018triangulated}.
	
	Now let us prove the other direction.
	Assume that $ev_{S^1,\Gm}(\mathcal{X})$ is very effective in the sense of \cite[Definition 3.5]{garkusha2018triangulated}.
	Then for every $n > 0$ the diagonal of the bisimplicial abelian group $\mathcal{X}(\Gmn{n})(\wh{\Delta}^{\bullet}_{K/k}) $ is contractible.
	
	We need to show that $\mathcal{X}$ satisfies Suslin's contractibility, i.e. that for every $U \in \Sm$, the diagonal of $\mathcal{X}(\Gmn{1} \times U)(\wh{\Delta}^{\bullet}_{K/k}) $ is contractible.
	So take $U \in \Sm$. Then the functor $\mathcal{X}(U \times -) : \Sm \rightarrow \Delta^{op}\ShvA$ is again a special enriched motivic $\Cor$-space, so we can form the framed bispectrum $ev_{S^1,\Gm}(\mathcal{X}(U \times -))$.
	Let $ev_{S^1,\Gm}(\mathcal{X}(U \times -))^f$ be a levelwise local fibrant replacement of $ev_{S^1,\Gm}(\mathcal{X}(U \times -))$.
	From \cite[Lemma 2.8]{garkusha2018triangulated} it follows that $ev_{S^1,\Gm}(\mathcal{X}(U \times -))^f$ is motivically fibrant.
	
	By Theorem \ref{ROLemma} we have an isomorphism
	$$ev_{S^1,\Gm}(\mathcal{X}) \wedge \Sigma^{\infty}_{S^1,\Gm}U_+ \cong  ev_{S^1,\Gm}(\mathcal{X}(U \times -))$$
	in $\SH(k)[1/p]$.
	So after inverting $p$, the bispectrum $ev_{S^1,\Gm}(\mathcal{X}(U \times -))^f$ is a motivically fibrant replacement of $ev_{S^1,\Gm}(\mathcal{X}) \wedge \Sigma^{\infty}_{S^1,\Gm}U_+$.
	
	Since both $ev_{S^1,\Gm}(\mathcal{X})$ and $\Sigma^{\infty}_{S^1,\Gm}U_+$ are very effective, this implies that
	$ev_{S^1,\Gm}(\mathcal{X}(U \times -))^f$ is very effective in $\SH(k)[1/p]$.
	
	From Lemma \ref{invertplemma} it now follows that $ev_{S^1,\Gm}(\mathcal{X}(U \times -))^f$ is very effective when regarded as an object in $\SH(k)$.
	With \cite[Lemma 3.2]{garkusha2018triangulated} it follows that
	the diagonal of $\mathcal{X}(\Gmn{1} \times U)(\wh{\Delta}^{\bullet}_{K/k}) $ is contractible, so $\mathcal{X}$ satisfies Suslin's contractibility.
\end{proof}

The proof of Proposition \ref{effcomparsion} also implies the following corollary.
\begin{cor}
	Let $\mathcal{X}$ be a special enriched motivic $\Cor$-space.
	Then $\mathcal{X}$ is very effective in the sense of Defintion \ref{suslindef} if and only if for every $n \geq 1$ the diagonal of
	$\mathcal{X}(\Gmn{n})(\widehat{\Delta}^\bullet_{K/k})$
	is contractible.
\end{cor}

Let $\mathcal{U} : DM_{\Cor} \rightarrow \SH(k)$ be the canonical forgetful functor, and let $\mathcal{L} : \SH(k) \rightarrow DM_{\Cor}$ be its left adjoint.
Let $DM_{\Cor}^{\eff}$ be the full triangulated subcategory of $DM_{\Cor}$ compactly generated by the set
$\{ \motive{U} \mid U \in \Smk\}.$
See \ref{Gmmotivedef} for the definition of $\motive{U}$.
Recall that $\SH^{\eff}(k)$ is the full subcategory of $\SH(k)$ generated by the suspension bispectra $\Sigma^{\infty}_{S^1,\Gm}U_+ $ for $U \in \Smk$.

\begin{lem}\label{lemcompactrestrict}
	Let $\cc C$ and $\cc D$ be triangulated categories, and
	let $F: \cc C \rightarrow \cc D$ be a triangulated functor. Assume that $F$ preserves small coproducts.
	Let $S_{\cc C}$ be a full triangulated subcategory of $\cc C$ compactly generated by a set $\Sigma_{\cc C}$.
	Let $S_{\cc D}$ be a full triangulated subcategory of $\cc D$ closed under small coproducts.
	Assume that for every $A \in \Sigma_{\cc C}$ we have $F(A) \in S_{\cc D}$. Then for every $A \in S_{\cc C}$ we have $F(A) \in S_{\cc D}$. In particular $F$ restricts to a triangulated functor
	$F : S_{\cc C} \rightarrow S_{\cc D}.$
\end{lem}
\begin{proof}
	Consider the full subcategory $F^{-1}(S_{\cc D})$ in $\cc C$ consisting of all those objects $A \in \cc C$ for which $F(A) \in S_{\cc D}$.
	We need to show that $S_{\cc C} \subseteq F^{-1}(S_{\cc D})$.
	Since $\Sigma_{\cc C} \subseteq F^{-1}(S_{\cc D})$, it suffices due to \cite[Theorem 2.1]{neeman1996grothendieckduality} to show that the subcategory $F^{-1}(S_{\cc D})$ is a triangulated subcategory closed under triangles and small coproducts in $\cc C$.
	
	If we have a triangle
	$$X \rightarrow Y \rightarrow Z \rightarrow \Sigma X$$
	in $\cc C$ with $X, Y \in F^{-1}(S_{\cc D})$, then
	$$F(X) \rightarrow F(Y) \rightarrow F(Z) \rightarrow \Sigma F(X) $$
	is a triangle in $\cc D$ with $F(X), F(Y) \in S_{\cc D}$. Since $S_{\cc D}$ is closed under triangles it follows that $F(Z) \in S_{\cc D}$, so $Z \in F^{-1}(S_{\cc D})$, so $F^{-1}(S_{\cc D})$ is closed under triangles.
	Since $F$ preserves small coproducts and $S_{\cc D}$ is closed under small coproducts, it follows that $F^{-1}(S_{\cc D})$ is closed under small coproducts.
	Therefore $F^{-1}(S_{\cc D})$ is closed under triangles and small coproducts. We get that $S_{\cc C} \subseteq F^{-1}(S_{\cc D})$, which proves the lemma.
\end{proof}

\begin{lem}\label{lemmaLeffgenerated}
	If $X \in \SH^{\eff}(k)$, then $\mathcal{L}(X) \in DM_{\Cor}^{\eff}$.
	So the functor $\mathcal{L} : \SH(k) \rightarrow DM_{\Cor}$ restricts to a functor
	$$\mathcal{L}^{\eff} :\SH^{\eff}(k) \rightarrow DM_{\Cor}^{\eff}.$$
\end{lem}
\begin{proof}
	By Lemma \ref{motivecomparison} we have $\mathcal{L}(\Sigma^{\infty}_{S^1,\Gm}U_+) \cong \motive{U} \in DM_{\Cor}^{\eff}$.
	Since the $\Sigma^{\infty}_{S^1,\Gm}U_+$ compactly generate $\SH^{\eff}(k)$ the result now follows from Lemma \ref{lemcompactrestrict}.
\end{proof}

\begin{lem}\label{lemUfilteredcolim}
	The triangulated functor $\mathcal{U} : DM_{\Cor} \rightarrow SH(k)$ preserves small coproducts.
\end{lem}
\begin{proof}
	Let $I$ be a set, and $\{ A_i \mid i \in I\}$ a family of objects. We want to show that the canonical morphism
	$$\underset{i \in I}{\coprod} \mathcal{U}(A_i) \rightarrow  \mathcal{U}(\underset{i \in I}{\coprod} A_i) $$
	is an isomorphism in $SH(k)$.
	The triangulated category $SH(k)$ is compactly generated by the set
	$$\Sigma_{SH(k)} := \{\Sigma_{S^1,\Gm}^{\infty}U_+ \wedge \Gmn{n} \mid U \in \Smk, n \in \bb Z\}. $$
	Thus to show that the above morphism is an isomorphism, it suffices to show that for all $G \in \Sigma_{SH(k)}$ that the map
	$$\mathrm{Hom}_{SH(k)}(G, \underset{i \in I}{\coprod} \mathcal{U}(A_i)) \rightarrow \mathrm{Hom}_{SH(k)}(G,  \mathcal{U}(\underset{i \in I}{\coprod} A_i)) $$
	is an isomorphism of abelian groups.
	
	The objects $\Sigma_{S^1,\Gm}^{\infty}U_+ \wedge \Gmn{n}$ are compact in $SH(k)$, and also each $\mathcal{L}(\Sigma_{S^1,\Gm}^{\infty}U_+ \wedge \Gmn{n})$ is compact in $DM_{\Cor}$.
	So for all $G \in \Sigma_{SH(k)}$ we get a chain of bijections
	$$\mathrm{Hom}_{SH(k)}(G, \underset{i \in I}{\coprod} \mathcal{U}(A_i)) \cong \underset{i \in I}{\coprod} \mathrm{Hom}_{SH(k)}(G,  \mathcal{U}(A_i)) \cong \underset{i \in I}{\coprod} \mathrm{Hom}_{SH(k)}(\mathcal{L}(G),  A_i) \cong$$$$
	\cong \mathrm{Hom}_{SH(k)}(\mathcal{L}(G),  \underset{i \in I}{\coprod}A_i) \cong \mathrm{Hom}_{SH(k)}(G,  \mathcal{U}(\underset{i \in I}{\coprod}A_i)).$$
	Therefore 
	$$\underset{i \in I}{\coprod} \mathcal{U}(A_i) \rightarrow  \mathcal{U}(\underset{i \in I}{\coprod} A_i) $$
	is an isomorphism in $SH(k)$, and $\mathcal{U}$ preserves small coproducts.
\end{proof}

\begin{lem} \label{effDMSH}
	Assume that $\Cor$ satisfies the $\widehat{\Delta}$-property in the sense of Definition \ref{framedCorDef}.
	Then for all $X \in DM_{\Cor}$ we have $X \in DM_{\Cor}^{\eff}$ if and only if $\mathcal{U}(X) \in \SH^{\eff}(k)$.
\end{lem}
\begin{proof}
	Our first claim is that $\mathcal{U}(\motive{U}) \in  \SH^{\eff}(k)$ for every $U \in \Smk$.
	
	Let $\mathbbm{1}_{\Cor} \in DM_{\Cor}$ be the monoidal unit. Then $$\mathcal{U} (\motive{U}) \cong \mathcal{U} (\motive{U} \wedge \mathbbm{1}_{\Cor}).$$
	We can regard $\SH^{\eff}(k)[1/p]$ as a full subcategory of $\SH^{\eff}(k)$.
	From Lemma \ref{invertplemma} it follows that the adjunction $\mathcal{U} : DM_{\Cor} \leftrightarrows \SH(k) : \mathcal{L}$ restricts to an adjunction  $$\mathcal{U} : DM_{\Cor} \leftrightarrows \SH(k)[1/p] : \mathcal{L}.$$
	By \cite[Appendix B, Corollary B.2]{levine2019algebraic} the suspension spectrum $\Sigma_{S^1,\Gm}U_+$ is strongly dualizable in $\SH(k)[1/p]$.
	So we can apply \cite[Lemma 4.6]{bachmann2020effectivity} to get an isomorphism 
	$$\mathcal{U} (\motive{U} \wedge \mathbbm{1}_{\Cor}) \cong \mathcal{U} (\mathcal{L}(\Sigma^{\infty}_{S^1,\Gm}U_+) \wedge \mathbbm{1}_{\Cor}) \cong  \Sigma_{S^1,\Gm}U_+ \wedge \mathcal{U}(\mathbbm{1}_{\Cor}) $$
	in $\SH(k)[1/p]$.
	Now $\Sigma_{S^1,\Gm}U_+$ is effective, and $\SH^{\eff}(k)$ is closed under the $\wedge$ product, so to show that $\mathcal{U} (\motive{U}) \in \SH^{\eff}(k)$ , we now just need to show that $\mathcal{U}(\mathbbm{1}_{\Cor}) \in \SH^{\eff}(k)$.
	
	The bispectrum $\mathcal{U}(\mathbbm{1}_{\Cor})$  is isomorphic to the bispectrum $\motive{pt} = (M_{\Cor}(\Gmn{j}))_{j \geq 0}$. By construction, the latter bispectrum is a framed bispectrum in the sense of \cite{garkusha2018triangulated}, because $\Cor$ has framed correspondences. Since $\Cor$ also has the $\widehat{\Delta}$-property, the bispectrum $\motive{pt}$ is effective in the sence of \cite[Definition 3.5]{garkusha2018triangulated}.
	And by \cite[Theorem 3.6]{garkusha2018triangulated} this implies that $\mathcal{U}(\mathbbm{1}_{\Cor}) \in \SH^{\eff}(k)$. 
	So we now have for every $U \in \Smk$ that $\mathcal{U}(\motive{U}) \in \SH^{\eff}(k)$.
	
	Due to Lemma \ref{lemUfilteredcolim} we can now apply Lemma \ref{lemcompactrestrict} to get for every $E \in DM^{\eff}_{\Cor}$ that $\mathcal{U}(E) \in SH^{\eff}(k)$ (This argument is similar to an argument used in the proof of \cite[Corollary 5.4]{bachmann2020effectivity}).
	So the functor $\mathcal{U} :  DM_{\Cor} \rightarrow \SH(k)$ restricts to a functor $$\mathcal{U}^{\eff} : DM_{\Cor}^{\eff} \rightarrow  \SH^{\eff}(k).$$
	This shows one direction of the lemma.
	Let us now show the other direction of the lemma.
	According to Lemma \ref{lemmaLeffgenerated} the functor $\mathcal{L} : \SH(k)\rightarrow DM_{\Cor}$ restricts to a functor $$\mathcal{L}^{\eff} : \SH^{\eff}(k)\rightarrow DM_{\Cor}^{\eff}.$$
	The functor $\mathcal{L}^{\eff}$ is left adjoint to $\mathcal{U}^{\eff}$.

	By \cite[Remark 2.1]{voevodsky2002open} the inclusion functors $\iota: DM_{\Cor}^{\eff} \rightarrow DM_{\Cor}$ and $\iota: \SH^{\eff}(k) \rightarrow \SH(k)$ have right adjoints $r_0 : DM_{\Cor} \rightarrow DM_{\Cor}^{\eff}$ and $r_0 : \SH(k)\rightarrow \SH^{\eff}(k)$. 
	
	The following diagrams commute:
	$$\xymatrix{  DM_{\Cor}^{\eff} \ar[d]_{\iota} &  \ar[l]_{\mathcal{L}^{\eff}}  \SH^{\eff}(k) \ar[d]^{\iota} & & DM_{\Cor}^{\eff} \ar[d]_{\iota} \ar[r]^{\mathcal{U}^{\eff}}  &   \SH^{\eff}(k) \ar[d]^{\iota}\\
		DM_{\Cor}  & \ar[l]_{\mathcal{L}} \SH(k) & & 
		DM_{\Cor} \ar[r]^{\mathcal{U}} &  \SH(k) }  $$
	From the commutativity of the left diagram it follows by adjunction that also the following diagram commutes:
	$$\xymatrix{  DM_{\Cor}^{\eff} \ar[r]^{\mathcal{U}^{\eff}} & \SH^{\eff}(k)  \\
		DM_{\Cor} \ar[u]^{r_0}\ar[r]^{\mathcal{U}}&\ar[u]_{r_0} \SH(k)} $$
	Take $X \in DM_{\Cor}$ such that $\mathcal{U}(X) \in \SH^{\eff}(k)$. We need to show that $X \in DM_{\Cor}^{\eff}$. Since $\mathcal{U}(X) \in \SH^{\eff}(k)$ the counit $\epsilon$ of the adjunction $\iota: \SH^{\eff}(k) \rightleftarrows \SH(k) : r_0$ is an isomorphism at $\mathcal{U}(X)$. So
	$$\epsilon_{\mathcal{U}(X)} :  \iota(r_0(\mathcal{U}(X))) \overset{\sim}{\rightarrow} \mathcal{U}(X) $$
	is an isomorphism in $\SH(k)$.
	By the commutativity of the above diagram this implies that the composite
	$$\mathcal{U}(\iota(r_0(X))) = \iota(\mathcal{U}^{\eff}(r_0(X))) \cong \iota(r_0(\mathcal{U}(X))) \overset{\sim}{\rightarrow} \mathcal{U}(X) $$ is an isomorphism in $\SH(k)$. But this composite is equal to $\mathcal{U}(\epsilon_X)$ where 
	$$\epsilon_X: \iota(r_0(X)) \rightarrow X $$ is the counit map of the adjunction $\iota: DM_{\Cor}^{\eff} \rightleftarrows DM_{\Cor} : r_0$.
	Now the forgetful functor $\mathcal{U} : DM_{\Cor} \rightarrow \SH(k)$ is conservative, so if $\mathcal{U}(\epsilon_X)$ is an isomorphism in $\SH(k)$, then also $\epsilon_X$ is an isomorphism in $DM_{\Cor}$.
	But this then implies that $X$ lies in $DM_{\Cor}^{\eff}$, which proves the lemma.
\end{proof}

We have an evaluation functor
$$ev_{\Gm} : \Ch([\Sm,\ShvA]) \rightarrow \SpGm(\Ch(\ShvA)).$$
For $\mathcal{X} \in [\Sm,\Delta^{op}\ShvA]$ we define
$$ev_{\Gm}(\mathcal{X}):= ev_{\Gm}(\Lambda(\mathcal{X})).$$

\begin{lem} \label{evforgetllemma}
	For $\mathcal{X} \in \mathrm{Spc}_{\Cor}[\Sm]$ we have a canonical isomorphism in $\SH(k)$
	$$\mathcal{U}(ev_{\Gm}(\mathcal{X})) \xrightarrow \sim ev_{S^1,\Gm}(\mathcal{X}).$$
\end{lem}
\begin{proof}
	Let $\bb Z^{S^n}$ be the reduced free simplicial abelian group on the pointed simplicial set $S^n$.
	The bispectrum $ev_{S^1,\Gm}(\mathcal{X}) = ev_{S^1,\Gm}(\EM(\mathcal{X}))$ can be computed in the $(n,m)$-th level as
	$$ev_{S^1,\Gm}(\mathcal{X})[n](m) = \bb Z^{S^n} \otimes \mathcal{X}(\Gmn{m})$$
	in $\cc M$.
	The bispectrum $\mathcal{U}(ev_{\Gm}(\mathcal{X})) = \mathcal{U}(ev_{\Gm}(\Lambda(\mathcal{X})))$ can be computed in the $(n,m)$-th level as
	$$\mathcal{U}(ev_{\Gm}(\mathcal{X}))[n](m) = DK(DK^{-1}(\mathcal{X})(\Gmn{m})[m])$$
	in $\cc M$.
	We claim that there is a natural homotopy equivalence
	$$DK(DK^{-1}(\mathcal{X})(\Gmn{m})[m]) \rightarrow \bb Z^{S^n} \otimes \mathcal{X}(\Gmn{m}) $$
	in $\cc M$.
	The chain complex $DK^{-1}(\bb Z^{S^n})$ is $\bb Z$ in degree $n$ and $0$ in all other degrees. It follows for every chain complex $A$ that
	$$A[n] \cong A \otimes DK^{-1}(\bb Z^{S^n}).$$
	According to \cite{nlab:eilenberg-zilber/alexander-whitney_deformation_retraction} the Dold-Kan correspondence preserves tensor products up to homotopy equivalence.
	We then get a homotopy equivalence
	$$DK(\Lambda(\mathcal{X})(\Gmn{m})[m]) \cong DK(DK^{-1}(\mathcal{X})(\Gmn{m})\otimes DK^{-1}(\bb Z^{S^n})) \rightarrow DK(DK^{-1}(\mathcal{X}(\Gmn{m})\otimes\bb Z^{S^n})) \cong$$$$\cong \mathcal{X}(\Gmn{m})\otimes\bb Z^{S^n}. 
	$$
	These maps assemble together into an isomorphism
	$\mathcal{U}(ev_{\Gm}(\mathcal{X})) \xrightarrow \sim ev_{S^1,\Gm}(\mathcal{X})$
	in $SH(k)$.
\end{proof}

Let $\mathrm{Spc}^{\veff}_{\Cor}[\Sm]$ be the full subcategory of $\mathrm{Spc}_{\Cor}[\Sm]$ consisting of the very effective special enriched motivic $\Cor$-spaces. By definition it is then also full subcategory of $\cc D([\Sm,\Delta^{op}\ShvA])$ consisting of the very effective special enriched motivic $\Cor$-spaces.

\begin{thm} \label{veffequivalence}
	Assume that $\Cor$ satisfies the $\widehat{\Delta}$-property in the sense of Definition \ref{framedCorDef}.
	Then the naive $\Gm$-evaluation functor induces an equivalence of categories
	$$ev_{\Gm} : \mathrm{Spc}^{\veff}_{\Cor}[\Sm] \rightarrow DM_{\Cor, \geq 0}^{\eff} .$$
	
\end{thm}
\begin{proof}
	By Theorem \ref{geqtheorem} we have an equivalence $$ev_{\Gm} : \mathrm{Spc}_{\Cor}[\Sm] \rightarrow DM_{\Cor, \geq 0} .$$
	So we just need to show for $\mathcal{X} \in \mathrm{Spc}_{\Cor}[\Sm]$ that
	$\mathcal{X} \in \mathrm{Spc}^{\veff}_{\Cor}[\Sm]$ if and only if $ev_{\Gm}(\mathcal{X}) \in DM_{\Cor, \geq 0}^{\eff} $.
	By Proposition \ref{effcomparsion} we know that $\mathcal{X} \in \mathrm{Spc}^{\veff}_{\Cor}[\Sm]$ if and only if $ev_{S^1,\Gm}(\mathcal{X}) \in SH^{fr}_{\nis}(k)$ is effective. By \cite[Theorem 3.6]{garkusha2018triangulated} this is the case if and only if $ev_{S^1,\Gm}(\mathcal{X})$ lies in $\SH^{\eff}(k)$.
	By Lemma \ref{evforgetllemma} we have a canonical isomorphism 
	$$ev_{S^1,\Gm}(\mathcal{X}) \cong \mathcal{U}(ev_{\Gm}(\mathcal{X}))  $$
	in $\SH(k)$.
	So $ev_{S^1,\Gm}(\mathcal{X}) \in \SH^{\eff}(k)$ if and only if $\mathcal{U}(ev_{\Gm}(\mathcal{X})) \in \SH^{\eff}(k)$ and by Lemma \ref{effDMSH} this is the case if and only if $ev_{\Gm}(\mathcal{X}) \in DM_{\Cor}^{\eff}$, which proves the theorem.
\end{proof}

\section{Reconstructing $SH^{\veff}(k)_{\bb Q}$}

In this section we apply the techniques and results from the previous sections to give new models for the stable motivic homotopy category of effective and very effective motivic bispectra with rational coefficients. It also requires the reconstruction theorem by \cite{garkusha2019compositio} and the theory Milnor-Witt correspondences \cite{bachmann2020effectivity,milnorwitt5,CF, deglise2020milnorwitt,milnorwitt4,milnorwitt6}.

Let $\widetilde{\mathrm{Cor}}$ be the category of finite Milnor-Witt correspondences in the sense of \cite{CF}.
Then $\widetilde{\mathrm{Cor}}$ is a strict $V$-category of correspondences satisfying the cancellation property (See \cite{milnorwitt4} for details). Furthermore it has framed correspondences by \cite{deglise2020milnorwitt}. It also satisfies the $\widehat{\Delta}$-property by \cite{bachmann2020effectivity}. 

Denote by $SH(k)_{\bb Q}$ the category of motivic bispectra $E$ whose sheaves of stable motivic homotopy groups $\pi_{*,*}^{\bb A^1}(E)$ are sheaves of rational vector spaces. The category $SH(k)_{\bb Q}$ is also called the \textit{rational stable motivic homotopy category}. It is the homotopy category of a stable model structure in which
weak equivalences are those morphisms of bispectra $f:E\to E'$ for which $\pi_{*,*}^{\bb A^1}(f)\otimes\bb Q$ is an isomorphism.
Let $\SH(k)_{\bb Q, \geq 0}$ be the full subcategory of $SH(k)_{\bb Q}$ on the connective objects. Here a bispectrum object $X \in SH(k)_{\bb Q}$ with rational stable $\bb A^1$-homotopy groups $\underline{\pi}_{p,q}^{\bb A^1}(X)\otimes\bb Q$ is called \textit{connective}, if $$\underline{\pi}_{p,q}^{\bb A^1}(X)\otimes\bb Q \cong 0$$
for all $p < q $.

Throughout this section we assume the base field $k$ to be perfect of characteristic different from 2.
The assumption on the characteristic is typical when working with finite Milnor--Witt correspondences.
A theorem of Garkusha~\cite[Theorem 5.5]{garkusha2019compositio} states that the forgetful functor
	$$\mathcal{U}: DM_{\widetilde{\mathrm{Cor}},\bb Q} \rightarrow \SH(k)_{\bb Q} $$
is an equivalence of categories. This theorem was actually proven under the assumption that $k$ is also infinite.
The latter assumption is redundant due to~\cite[A.27]{DKO} saying that the main result of~\cite{GP4} about strict invariance
for Nisnevich sheaves with framed transfers is also true for finite fields. 

\begin{dfn}
	We define $\mathrm{Spc}_{\widetilde{\mathrm{Cor}}, \bb Q}[\Sm]$, respectively $DM_{\widetilde{\mathrm{Cor}},\bb Q, \geq 0}$ to be the category $\mathrm{Spc}_{\Cor}[\Sm]$, respectively $DM_{\Cor, \geq 0}$, associated to the category of correspondences $\Cor = \widetilde{\mathrm{Cor}} \otimes \bb Q$.
	We call $\mathrm{Spc}_{\widetilde{\mathrm{Cor}}, \bb Q}[\Sm]$ the category of \textit{rational enriched motivic $\widetilde{\mathrm{Cor}}$-spaces}.
\end{dfn}

The following theorem says that the special rational enriched motivic $\widetilde{\mathrm{Cor}}$-spaces recover $\SH(k)_{\bb Q, \geq 0}$. 

\begin{thm}\label{thmSHgeq}
	The $(S^1,\Gm)$-evaluation functor is an equivalence of categories
	$$ev_{S^1,\Gm}: \mathrm{Spc}_{\widetilde{\mathrm{Cor}}, \bb Q}[\Sm] \rightarrow \SH(k)_{\bb Q,\geq 0}.$$
\end{thm}

\begin{proof}
	By Theorem \ref{geqtheorem} the $\Gm$-evalulation functor is an equivalence of categories
	$$ev_{\Gm}: \mathrm{Spc}_{\widetilde{\mathrm{Cor}}, \bb Q}[\Sm] \rightarrow DM_{\widetilde{\mathrm{Cor}},\bb Q, \geq 0} .$$
	By \cite[Theorem 5.5]{garkusha2019compositio} the forgetful functor
	$$\mathcal{U}: DM_{\widetilde{\mathrm{Cor}},\bb Q} \rightarrow \SH(k)_{\bb Q} $$
	is an equivalence of categories, and this implies that the forgetful functor
	$$\mathcal{U}: DM_{\widetilde{\mathrm{Cor}},\bb Q, \geq 0} \rightarrow \SH(k)_{\bb Q, \geq 0} $$
	is an equivalence of categories.
	So by Lemma \ref{evforgetllemma} the $(S^1,\Gm)$-evaluation functor
	$$ev_{S^1,\Gm}: \mathrm{Spc}_{\widetilde{\mathrm{Cor}}, \bb Q}[\Sm] \rightarrow \SH(k)_{\bb Q,\geq 0}$$
	is an equivalence of categories.
\end{proof}
Let $\SH^{\veff}(k)_{\bb Q}$ be the full subcategory of $\SH(k)_{\bb Q}$ on the very effective bispectra. Here an object $X \in SH(k)_{\bb Q}$ is said to be \textit{very effective} if it is both effective and connective:
$$\SH^{\veff}(k)_{\bb Q} = \SH^{\eff}(k)_{\bb Q} \cap \SH(k)_{\bb Q,\geq 0} .$$

\begin{dfn}
	We define $\mathrm{Spc}^{\veff}_{\widetilde{\mathrm{Cor}}, \bb Q}[\Sm]$, respectively $DM_{\widetilde{\mathrm{Cor}},\bb Q,\geq 0}^{\eff}$, to be the category $\mathrm{Spc}^{\veff}_{\Cor}[\Sm]$, respectively $DM_{\Cor,\geq 0}^{\eff}$, associated to the category of correspondences $\Cor = \widetilde{\mathrm{Cor}} \otimes \bb Q$.
	We call $\mathrm{Spc}^{\veff}_{\widetilde{\mathrm{Cor}}, \bb Q}[\Sm]$ the category of \textit{very effective rational enriched motivic $\widetilde{\mathrm{Cor}}$-spaces}.
\end{dfn}

We finish the paper with the following result stating that very effective rational enriched motivic  $\widetilde{\mathrm{Cor}}$-spaces recover $\SH^{\veff}(k)_{\bb Q}$.

\begin{thm}\label{finish}
	The $(S^1,\Gm)$-evaluation functor is an equivalence of categories
	$$ev_{S^1,\Gm}: \mathrm{Spc}^{\veff}_{\widetilde{\mathrm{Cor}}, \bb Q}[\Sm] \rightarrow \SH^{\veff}(k)_{\bb Q}.$$
\end{thm}
\begin{proof}
	By Theorem \ref{thmSHgeq} the $(S^1,\Gm)$-evaluation functor is an equivalence of categories
	$$ev_{S^1,\Gm}: \mathrm{Spc}_{\widetilde{\mathrm{Cor}}, \bb Q}[\Sm] \rightarrow \SH(k)_{\bb Q,\geq 0}.$$
	We want to show that it restricts to an equivalence of categories
	$$ev_{S^1,\Gm}: \mathrm{Spc}^{\veff}_{\widetilde{\mathrm{Cor}}, \bb Q}[\Sm] \rightarrow \SH^{\veff}(k)_{\bb Q}.$$
	For this we just need to show that a special enriched motivic $\Cor$-space $\mathcal{X}$ is very effective if and only if $ev_{S^1,\Gm}(\mathcal{X})$ is very effective in $SH(k)$.
	
	According to Proposition \ref{effcomparsion} the special enriched motivic $\Cor$-space $\mathcal{X}$ is very effective if and only if the framed bispectrum $ev_{S^1,\Gm}(\mathcal{X})$ is effective in $SH(k)^{fr}_{\nis}$. By \cite[Theorem 3.6]{garkusha2018triangulated} this is the case if and only if $ev_{S^1,\Gm}(\mathcal{X})$ is effective in $SH(k)$.
	This concludes the proof of the theorem.
\end{proof}


\begin{thebibliography}{10}

\bibitem{AlGG} H. Al Hwaeer, G. Garkusha, {{Grothendieck categories of enriched functors}},
    {J. Algebra} 450 (2016), 204--241.

	\bibitem{benjamin2017homotopy}
	B. Antieau, E. Elmanto, A primer for unstable motivic homotopy theory,
Proc. Sympos. Pure Math. 95,
American Mathematical Society, Providence, RI, 2017, pp. 305--370.

	
	\bibitem{bachmann2020effectivity} T. Bachmann, J. Fasel, On the effectivity of spectra representing motivic
	cohomology theories,
	in Milnor-Witt Motives
	(T. Bachmann, B. Calm\`es, F.~D\'eglise, J. Fasel, P.~A.~{\O}stv{\ae}r, eds.), Mem. Amer. Math. Soc., to appear.
	
		\bibitem{bayeh2015lefttransfer}
	 M. Bayeh, K. Hess, V. Karpova, M. Kedziorek, E. Riehl, B. Shipley, Left-induced model structures and diagram categories, Contemp. Math. 641 (2015), 49--81.
		
	\bibitem{bonart2022paper1}
	P. Bonart,
	\newblock Triangulated categories of big motives via enriched functors, preprint, 2023.
	
	\bibitem{BFgamma} A.~K. Bousfield, E.~M. Friedlander, Homotopy theory of $\Gamma$-spaces, spectra, and bisimplicial sets,
      In Geometric applications of homotopy theory ({P}roc. {C}onf., {E}vanston, {I}ll., 1977), {II}, Lecture Notes in Mathematics, Vol. 658, Springer-Verlag, 1978, pp. 80-130.	
	
	\bibitem{milnorwitt5} B. Calm\`es, J. Fasel, A comparison theorem for Milnor-Witt motivic cohomology,
	in Milnor-Witt Motives
	(T. Bachmann, B. Calm\`es, F.~D\'eglise, J. Fasel, P.~A.~{\O}stv{\ae}r, eds.), Mem. Amer. Math. Soc., to appear.
	
	
	\bibitem{CF} B. Calm\`es, J. Fasel, The category of finite Milnor-Witt correspondences,
	in Milnor-Witt Motives
	(T. Bachmann, B. Calm\`es, F.~D\'eglise, J. Fasel, P.~A.~{\O}stv{\ae}r, eds.), Mem. Amer. Math. Soc., to appear.
	
	\bibitem{cegarra2005diagonal}
	A. Cegarra, J. Remedios, The relationship between the diagonal and the bar constructions on a bisimplicial set, Topology Appl. 153(1) (2005), 21--51.


\bibitem{clarke2019grothendieck}
A. Clarke, Grothendieck categories, MSc project, Swansea University, 2019.
	
	
	\bibitem{deglise2020milnorwitt} F. ~D\'eglise, J. Fasel, Milnor-Witt motivic complexes,
	in Milnor-Witt Motives
	(T. Bachmann, B. Calm\`es, F.~D\'eglise, J. Fasel, P.~A.~{\O}stv{\ae}r, eds.), Mem. Amer. Math. Soc., to appear.
	
	\bibitem{dold1961diagonal}
	A. Dold, D. Puppe, Homologie nicht-additiver Funktoren, Anwendungen, Ann. Inst. Fourier 11 (1961), 201--312.
	
	\bibitem{DKO} A. Druzhinin, H. Kolderup, P.~A. \O stv\ae r, Strict $\mathbb A^1$-invariance over the integers,
               preprint arXiv:2012.07365.

\bibitem{dundas2003enriched}
B. Dundas, O. R{\"o}ndigs, P. {\O}stv{\ae}r,
\newblock {Enriched functors and stable homotopy theory,}
\newblock { Doc. Math.} 8 (2003), 409--488.


\bibitem{dundas2003motivic}
B. Dundas, O. R{\"o}ndigs, P. {\O}stv{\ae}r,
\newblock {Motivic functors,}
\newblock { Doc. Math.} 8 (2003), 489--525.

\bibitem{milnorwitt4} J. Fasel, P.~A.~{\O}stv{\ae}r, A cancellation theorem for Milnor-Witt correspondences,
in Milnor-Witt Motives
(T. Bachmann, B. Calm\`es, F.~D\'eglise, J. Fasel, P.~A.~{\O}stv{\ae}r, eds.), Mem. Amer. Math. Soc., to appear.

\bibitem{milnorwitt6} J. Fasel, F. F.~D\'eglise, The Milnor-Witt motivic ring spectrum and its associated
theories,
in Milnor-Witt Motives
(T. Bachmann, B. Calm\`es, F.~D\'eglise, J. Fasel, P.~A.~{\O}stv{\ae}r, eds.), Mem. Amer. Math. Soc., to appear.

\bibitem{garkusha2019compositio}
G. Garkusha,
\newblock {Reconstructing rational stable motivic homotopy theory},
Compos. Math. 155(7) (2019), 1424--1443. 
	
	\bibitem{garkusha2019derived}
	G. Garkusha, D. Jones,
	\newblock {Derived categories for Grothendieck categories of enriched
		functors},
	\newblock { Contemp. Math} 730 (2019), 23--45.
	
\bibitem{GP4} G. Garkusha, I. Panin, Homotopy invariant presheaves with framed transfers, Cambridge J. Math. 8(1) (2020), 1-94.
	
\bibitem{garkusha2021framedmotives}
   G. Garkusha, I. Panin, Framed motives of algebraic varieties (after V. Voevodsky), J. Amer. Math. Soc. 34(1) (2021), 261--313. 
   
	\bibitem{garkusha2018triangulated}
	G. Garkusha, I. Panin,
	\newblock {The triangulated categories of framed bispectra and framed motives},
	Algebra i Analiz 34(6) (2022), 135--169.
	
	\bibitem{garkusha2019framed}
	G. Garkusha, I. Panin, P. {\O}stv{\ae}r,
	\newblock {Framed motivic $\Gamma$-spaces},
	Izv. Math. 87(1) (2023), 1--28.


\bibitem{hirschhorn2003model}
P. S. Hirschhorn,
\newblock { Model categories and their localizations},
\newblock  American Mathematical Society, Providence, RI,  2003.

\bibitem{hovey2007model}
M. Hovey,
\newblock { Model categories},
\newblock American Mathematical Society, Providence, RI, 1999.
	
	
	\bibitem{levine2019algebraic}
	M. Levine, Y. Yang, G. Zhao, J. Riou,
	\newblock {Algebraic elliptic cohomology theory and flops I},
	\newblock { Math. Ann.} 375(3) (2019), 1823--1855.
	
	
	


\bibitem{nlab:eilenberg-zilber/alexander-whitney_deformation_retraction}
{nLab authors},
\newblock {{E}}ilenberg-{{Z}}ilber/{{A}}lexander-{{W}}hitney deformation
retraction,
\newblock
\url{https://ncatlab.org/nlab/show/Eilenberg-Zilber%2FAlexander-Whitney+deformation+retraction},
	\newblock
	\href{https://ncatlab.org/nlab/revision/Eilenberg-Zilber%2FAlexander-Whitney+deformation+retraction/3}{Revision
		3},
	(2023).
	
	
	\bibitem{neeman1996grothendieckduality}
	A. Neeman, The Grothendieck duality theorem via Bousfield’s techniques and Brown representability, J. Amer.
	Math. Soc. 9(1) (1996), 205--236.

	
	\bibitem{rondigs2008modules}
	O. R{\"o}ndigs, P. {\O}stv{\ae}r,
	\newblock {Modules over motivic cohomology},
	\newblock {Adv. Math.} 219(2) (2008), 689--727.
	
	\bibitem{S} G. Segal, Categories and cohomology theories, Topology 13 (1974), 293--312.
	
	\bibitem{voevodsky2002open}
	V. Voevodsky,
	\newblock Open problems in the motivic stable homotopy theory I,
Int. Press Lect. Ser. 3, I,
International Press, Somerville, MA, 2002, pp. 3--34.
	
\end{thebibliography}
\end{document}